\documentclass[12pt, twoside]{article}
\usepackage{amsfonts}
\usepackage{xcolor}
\usepackage{mathrsfs}
\usepackage[all,cmtip]{xy}
\usepackage{amsmath}
\usepackage{amssymb,amsthm,upref,amscd}
\usepackage{setspace}
\usepackage{enumerate}
\usepackage[titletoc]{appendix}
\usepackage{times}
\usepackage{cite}
\usepackage{tikz}
\usepackage[colorinlistoftodos]{todonotes}
\usetikzlibrary{arrows}
\numberwithin{equation}{section}

\def\titlerunning#1{\gdef\titrun{#1}}
\makeatletter
\def\author#1{\gdef\autrun{\def\and{\unskip, }#1}\gdef\@author{#1}}

\makeatother

\def\keywords#1{\par
\noindent\textbf{Keywords.} #1}

\allowdisplaybreaks

\theoremstyle{plain}
\newtheorem{Thm}{Theorem}[section]
\newtheorem{Lem}[Thm]{Lemma}

\newtheorem{Cor}[Thm]{Corollary}
\newtheorem{Prop}[Thm]{Proposition}
\newtheorem*{Thm*}{Theorem}
\newtheorem*{claim*}{Claim}
\theoremstyle{definition}

\newtheorem*{Def*}{Definition}
\newtheorem*{Cor*}{Corollary}
\newtheorem{Rem}[Thm]{Remark}

\newcommand{\equ}{equation}
\newcommand{\C}{\mathbb{C}}
\newcommand{\N}{\mathbb{N}}
\newcommand{\R}{\mathbb{R}}
\newcommand{\Z}{\mathbb{Z}}

\DeclareMathOperator{\spec}{spec}

\DeclareMathOperator{\vol}{vol} \DeclareMathOperator{\Real}{Re}

\newcommand\eps{\varepsilon}

\let\nhatoksa=\theenumi
\let\nhatoksb=\labelenumi
\let\nhatoksc=\theenumii
\let\nhatoksd=\labelenumii
\newlength{\nhalengtha}
\newlength{\nhalengthb}
\newlength{\nhalengthc}
\newcommand{\resetenum}{
\let\theenumi=\nhatoksa
\let\labelenumi=\nhatoksb
\let\theenumii=\nhatoksc
\let\labelenumii=\nhatoksd
\setlength{\leftmargini}{\nhalengtha}
\setlength{\leftmarginii}{\nhalengthb}
\setlength{\labelwidth}{\nhalengthc}
}

\newcommand\cc{\mathcal{C}}

\newcommand\ce{\mathcal{E}}
\newcommand\cf{\mathcal{F}}

\newcommand\ch{\mathcal{H}}

\newcommand\cj{\mathcal{J}}
\newcommand\ck{\mathcal{K}}
\newcommand\cl{\mathcal{L}}
\newcommand\cm{\mathcal{M}}
\newcommand\cn{\mathcal{N}}

\newcommand\cp{\mathcal{P}}
\newcommand\cq{\mathcal{Q}}
\newcommand\rr{\mathcal{R}}
\newcommand\cs{\mathcal{S}}

\def\mbs{\mathbb{S}}

\def\ig{\textit{g}}

\def\psitil{\widetilde{\psi}}

\def\ehat{\widehat{E}}

\def\ov{\overline}

\def\pa {\partial}

\def\op{\oplus}

\def\De{\Delta}

\def\al{\alpha}
\def\bt{\beta}
\def\be{\beta}
\def\de{\delta}
\def\Ga{\Gamma}
\def\ga{\gamma}
\def\la{\lambda}
\def\lm{\lambda}
\def\La{\Lambda}
\def\om{\omega}
\def\Om{\Omega}
\def\sa{\sigma}
\def\si{\sigma}
\def\vr{\varepsilon}
\def\va{\varphi}

\def\span{\hbox{span}}

\def\vol{\mathrm{vol}}
\def\real{\hbox{Re}}

\def\tce{\widetilde\ce}
\def\te{\widetilde E}
\newcommand{\inp}[2]{\left\langle#1,#2\right\rangle}

\begin{document}

\baselineskip=17pt

\titlerunning{A spinorial analogue of Brezis-Nirenberg}

\title{A spinorial analogue of the Brezis-Nirenberg theorem involving the critical Sobolev exponent}

\author{Thomas Bartsch,\, Tian Xu\footnote{Supported by the National Science Foundation of China (NSFC 11601370, 11771325)
and the Alexander von Humboldt Foundation of Germany}}

\date{}

\maketitle

\begin{abstract}
Let $(M,\ig,\sa)$ be a compact Riemannian spin manifold of dimension $m\geq2$, let $\mbs(M)$ denote the spinor bundle on $M$, and let $D$ be the Atiyah-Singer Dirac operator acting on spinors $\psi:M\to\mbs(M)$. We study the existence of solutions of the nonlinear Dirac equation with critical exponent
\[
  D\psi = \la\psi + f(|\psi|)\psi + |\psi|^{\frac2{m-1}}\psi \tag{NLD}
\]
where $\la\in\R$ and $f(|\psi|)\psi$ is a subcritical nonlinearity in the sense that $f(s)=o\big(s^{\frac2{m-1}}\big)$ as $s\to\infty$. A model nonlinearity is $f(s)=\al s^{p-2}$ with $2<p<\frac{2m}{m-1}$, $\al\in\R$. In particular we study the nonlinear Dirac equation
\[
  D\psi=\la\psi+|\psi|^{\frac2{m-1}}\psi, \quad \la\in\R. \tag{BND}
\]
This equation is a spinorial analogue of the Brezis-Nirenberg problem. As corollary of our main results we obtain the existence of least energy solutions $(\la,\psi)$ of (BND) and (NLD) for every $\la>0$, even if $\la$ is an eigenvalue of $D$. For some classes of nonlinearities $f$ we also obtain solutions of (NLD) for every $\la\in\R$, except for non-positive eigenvalues. If $m\not\equiv3$ (mod 4) we obtain solutions of (NLD) for every $\la\in\R$, except for a finite number of non-positive eigenvalues.  In certain parameter ranges we obtain multiple solutions of (NLD) and (BND), some near the trivial branch, others away from it.

The proofs of our results are based on variational methods using the strongly indefinite energy functional associated to (NLD).\\

\noindent{\bf MSC 2010:} Primary: 53C27; Secondary: 35R01

\keywords{nonlinear Dirac equation; critical Sobolev exponent; strongly indefinite functional; variational methods; bifurcation}


\end{abstract}

\newpage

\tableofcontents

\section{Introduction}
Let $(M,\ig)$ be an $m$-dimensional compact manifold. We assume that $M$ is spin, and we fix a spin structure $\sa$ on $M$. We denote by $\mbs(M)=Spin(TM)\times_\rho\mbs_m$ the spinor bundle on $M$ with hermitian metric $(\cdot,\cdot)$ and compatible spin connection $\nabla^\mbs$. The Clifford multiplication
\[
  TM\otimes \mbs(M)\to\mbs(M)
\]
is denoted by $X\otimes \psi\mapsto X\cdot \psi$. Let $D=D_M$ be the (Atiyah-Singer) Dirac operator defined on $\Ga(\mbs(M))$, i.e.\
$D=\sum_{k=1}^m e_k\cdot\nabla^\mbs_{e_k}$ for a local orthonormal frame $\{e_1,\dots,e_m\}$ of $TM$.

In this paper, we want to find solutions $\psi\in\Ga(\mbs(M))$ of the nonlinear Dirac equation
\begin{\equ}\label{main equ}
  D\psi = \la\psi + f(|\psi|)\psi + |\psi|^{\frac{2}{m-1}}\psi, \tag{NLD}
\end{\equ}
where $f:[0,+\infty)\to\R$ satisfies $f(0)=0$, $f(s)/s^{\frac{2}{m-1}}\to0$ as $s\to\infty$. In particular, $f(s)$ grows subcritically for $s\to\infty$. An important special case is $f\equiv0$. The exponent $2^*:=\frac{2m}{m-1}=2+\frac{2}{m-1}$ is critical here because the form domain $H^{\frac12}(M,\mbs(M))$ of the Dirac operator embeds into $L^q(M,\mbs(M))$ for $1\le q\le2^*$, and the embedding is compact precisely if $q<2^*$. This is closely related to the fact that on $\R^m$ with Dirac operator $D_{\R^m}$ the equation
\[
  D_{\R^m}\psi=|\psi|^{2^*-2}\psi
\]
has a family $\psi_\eps$ of solutions given by
\[
  \psi_\eps(x)=\eps^{2^*-2}\psi_1(\eps x)\quad\text{with}\quad \psi_1(x)=\frac{1}{(1+|x|^2)^{(m-1)/2}}(1-x)\cdot\psi_0
\]
where $\psi_0\in\mbs_m$ is a spinor with $|\psi_0|=\frac{m^{(m-1)/2}}{\sqrt{2}}$; here the dot ``$\cdot$" denotes the Clifford multiplication of an element of the Clifford algebra $Cl(\R^m)$ and a spinor.

Nonlinear Dirac equations on space-time $\R^4$ with various types of nonlinearities have been investigated in \cite{Esteban-Sere:1995} or \cite{Bartsch-Ding:2006b}, for instance. We refer the reader to the surveys \cite{Esteban-Sere:2002, Esteban-Lewin-Sere:2008} for references to the literature. Results about nonlinear Dirac equations on spin manifolds, motivated by geometry or physics, can be found in \cite{Ammann, AGHM, AHM, Branding:2016, Branding:2019, Chen-Jost-Wang:2008,  Grosse:2012,Isobe:2011, Isobe:2013, Isobe:2015, Isobe:2017-1, Isobe:2017-2, Raulot:2009}. 

Concerning problem \eqref{main equ} on a compact spin manifold the case $\la=0$ and $f=0$ is of particular geometric relevance, called the spinorial Yamabe-type equation. A nontrivial solution $\psi_0$ of
\begin{equation}\label{eq:geom}
  D\psi_0 = |\psi_0|^{2^*-2}\psi_0
\end{equation}
leads to a generalized metric $\ig_0=\left(\frac{|\psi_0|}{\|\psi_0\|_{2^*}}\right)^{4/(m-1)}\ig$ in the conformal class $[\ig]$ of $\ig$. It is a metric if $|\psi_0|>0$. There holds $\vol_{\ig_0}(M)=1$ and the {\it B\"ar-Hijazi-Lott invariant}
\[
  \la_{\min}^+(M,[\ig],\sa) = \inf_{\widetilde\ig\in[\ig]}\la_1^+(\widetilde\ig)\vol_{\widetilde\ig}(M)^{\frac1m}
\]
is achieved by $\ig_0$. Here $\la_1^+(\widetilde\ig)$ is the smallest positive eigenvalue of the associated Dirac operator $D_{\widetilde\ig}$ on $(M,\widetilde\ig,\si)$. Equivalently, the functional
\[
  J(\psi) = \frac{\left(\int_M|D\psi|^{\frac{2m}{m+1}}\,d\vol_\ig\right)^{\frac{m+1}{m}}}{\int_M(D\psi,\psi)\,d\vol_\ig}
\]
achieves its infimum at $\psi_0$ and $J(\psi_0) = \la_1^+(\ig_0) = \la_{\min}^+(M,[\ig],\sa)$ where the infimum is taken over the set $\Ga^+(\mbs(M))$ of all smooth spinor fields with $\int_M((D\psi,\psi)\,d\vol_\ig > 0$.

In \cite{Ammann2003, AGHM, AHM} it was shown that
\begin{\equ}\label{spinorial Yamabe ineq}
  \la_{\min}^+(M,[\ig],\si) \le \la_{\min}^+(S^m) = \frac{m}2\om_m^{\frac1m}
\end{\equ}
where $\la_{\min}^+(S^m)=\la_{\min}^+(S^m,[\ig_{S^m}],\si)$ denotes the B\"ar-Hijazi-Lott invariant for the standard sphere $S^m$ equipped with the canonical metric and the unique spin structure, and $\om_m$ stands for the volume of $S^m$. It is known that $\la_{\min}^+(S^m)=\frac{m}{2}\om_m^{1/m}$ is achieved. Moreover, $\la_{\min}^+(M,[\ig],\si)$ is achieved if the strict inequality in \eqref{spinorial Yamabe ineq} holds. We refer the reader to \cite{Ammann,AGHM,Bar,Ginoux,Hijazi,Lott} for these results.

The spinorial analogue of the Brezis-Nirenberg equation
\begin{equation}\label{eq:BND}
  D\psi = \la\psi + |\psi|^{2^*-2}\psi \tag{BND}
\end{equation}
has been treated by Isobe \cite{Isobe:2011}. The energy functional
\begin{equation}\label{ce lm}
  \ce_\la(\psi) = \frac12\int_M(D\psi,\psi)d\vol_\ig - \frac\la2\int_M|\psi|^2d\vol_\ig - \frac1{2^*}\int_M|\psi|^{2^*}d\vol_\ig.
\end{equation}
associated to \eqref{eq:BND} is strongly indefinite because the spectrum $\spec(D)$ consists of an infinite sequence of eigenvalues \
$\ldots<\la_{-1}<\la_0\le0<\la_1<\la_2<\dots$ \ with $|\la_k|\to\infty$ as $|k|\to\infty$. Consequently, a critical point of $\ce_\la$ has infinite Morse index and infinite co-index. In order to avoid this indefiniteness Isobe used a dual variational principle. Then he could apply the classical mountain pass theorem provided $m\ge4$, $\la\notin\spec(D)$, and $\la>0$.

In the present paper we deal with the more general equation \eqref{main equ} and present a different variational approach that works in all dimensions $m\ge2$ and also when $\la\in\spec(D)$. We do not use a dual functional but work instead with the strongly indefinite functional
\begin{equation}\label{the functional}
  \cl_\la(\psi)
   = \frac12\int_M(D\psi,\psi)d\vol_\ig - \frac\la2\int_M|\psi|^2d\vol_\ig - \int_M F(|\psi|)d\vol_\ig - \frac1{2^*}\int_M|\psi|^{2^*}d\vol_\ig
\end{equation}
where $F(s):=\int_0^s f(t)t\,dt$. Equation \eqref{main equ} is the Euler-Lagrange equation associated to \eqref{the functional}. We show that $\cl_\la$ satisfies the Palais-Smale condition below a critical value $\ga_{crit}$. Then we minimize $\cl_\la$ on the Nehari-Pankov manifold $\cp_\la$ if $\la_{k-1}<\la<\la_k$. In order to prove $\ga(\la):=\inf_{\cp_\la}\cl_\la<\ga_{crit}$ we construct suitable test spinors $\bar\va_\eps$ using the Bourgignon-Gauduchon trivialization. Since these do not lie on $\cp_\la$ we have to find modifications of the test spinors that lie on the Nehari-Pankov manifold, and we have to control the energy of these modifications. This is the main technical difficulty that we have to overcome, in particular when $\la\in\spec(D)$. In that case we use a new idea, replacing the test spinors $\bar\va_\eps$ by $\bar\va_\eps-T(\bar\va_\eps)$ where $T(\psi)\in\ker(D-\la)$ is the nearest neighbor of $\psi$ in $\ker(D-\la)$ with respect to the $L^{2^*}$ norm. Observe that $T$ is a nonlinear projection.

The minimization argument yields a least energy spinor $\psi_\la$ solving \eqref{main equ} provided $\la\in\spec(D)\setminus\{\la_k:k\le0\}$. For $k\ge1$ and $\la<\la_k$ close to $\la_k$ we obtain a second solution $\psitil_\la$ using a new min-max scheme. Essentially, we minimize $\cl_\la$ on a suitably constructed submanifold of the Nehari-Pankov manifold $\cp_\la$ of codimension $d_k:=\dim\ker(D-\la_k)$. The energy of the solution $\psitil_\la$ satisfies $\cl_\la(\psitil_\la)\to\cl_{\la_k}(\psi_{\la_k})>0$ as $\la\nearrow\la_k$, hence it may be considered as a continuation of the least energy solution $\psi_\la$, $\la_k\le\la<\la_{k+1}$. Clearly $\psitil_\la$ differs from $\psi_\la$ because $\cl_\la(\psi_\la)\to0$ as $\la\nearrow\la_k$.
We would like to mention that standard bifurcation theory for potential operators yields $d_k$ pairs of spinors $\pm\psi_{\la,j}$ solving \eqref{main equ} for $\la<\la_k$ close to $\la_k$. The solution $\psitil_\la$ cannot be obtained in this way. The construction of $\psitil_\la$ is new and can be generalized to other parameter-dependent variational problems.

As a last result we provide a uniform bound $\nu>0$ such that the solutions that bifurcate from $\la_k$ continue to exist for $\la\in(\la_k-\nu,\la_k)$. The Weyl formula for the Dirac operator implies that the number of  solutions of \eqref{main equ} and \eqref{eq:BND} becomes arbitrarily large as $|\la|\to\infty$, provided $m\not\equiv3$ (mod 4).

Setting
\[
  \cs_k:=\{(\la,\cl_\la(\psi_\la)):\la_{k-1}\le\la<\la_k\} \cup \{(\la,\cl_\la(\psitil_\la)):\la<\la_{k-1}\text{ close to $\la_{k-1}$}\}
\]
we can visualize the energy branches bifurcating from $\la_k$ and $\la_{k+1}$ in Figure~1.

\begin{figure}[ht]
 \centering
\begin{tikzpicture}[scale=1.3]
\draw[thick,->,>=stealth,line width=0.2ex] (-1,0) -- (9.5,0); 
\draw[thick,->,>=stealth,line width=0.2ex] (0,0) -- (0,3.5);
\draw[dashed,line width=0.25ex] (-1,2.3) -- (9.5,2.3); 
\draw[loosely dotted,line width=0.25ex] (6,0) -- (6,2.3); 
\coordinate [label=135: $\ga_{crit}$] (cc) at (0,2.3);
\draw (cc);
\coordinate [label=-90:$0$] (o) at (0,-0.15);
\draw (o);
\coordinate [label=-90:$\cdots$] (z) at (2,-0.15);
\draw (z);
\coordinate [label=-90:$ \lm_{k}-\nu$] (a) at (4.5,-0.1); 
\draw (a);
\draw[loosely dotted,line width=0.25ex] (4.5,0) -- (4.5,2.3);
\coordinate [label=-90:$ \lm_{k}$] (b) at (6,-0.1); 
\draw (b);
\coordinate [label=-90:$ \lm_{k+1}$] (c) at (8,-0.1); 
\draw (c);
\coordinate [label=-90:$\cdots$] (zz) at (9,-0.15); 
\draw (zz);
\coordinate [label=0:$\lm$] (e) at (9.5,0); 
\draw (e);
\coordinate [label=90:$\text{Energy}$] (f) at (0,3.5);
\draw (f);
\foreach \p in {a, b, c} \draw (\p)++(0,0.1) circle (2pt); 
\foreach \p in {cc} \fill (\p) circle (2pt);
\draw[style=dotted,line width=1.5pt] (3.6,1.2) -- (3.8, 1.1); 
\draw[line width=1.5pt]  (3.8,1.1) .. controls (4,1) 
.. (4.6,0.6) .. controls (5.1,0.3) and (5.6,0.2) .. (6, 0);
\draw[style=dotted,line width=1.5pt] (3.6,2.1) -- (3.8, 2.05); 
\draw[line width=1.5pt]  (3.8,2.05) .. controls (4,2) 
.. (4.6,1.85) .. controls (5.1,1.5) and (5.6,0.8) .. (6, 0);
\coordinate [label=-180: $\cs_{k}$] (sk) at (3.7,1.1);
\draw (sk);
\draw[style=dotted,line width=1.5pt] (5.3,2.1) -- (5.5, 1.95); 
\draw[line width=1.5pt] (5.5, 1.95) .. controls (5.9, 1.85) .. (6,1.8); 
\draw[line width=1.5pt] (6,1.8) .. controls (6.6,1.5) 
.. (6.8,1.3) .. controls (7,1.1)
.. (7.2,0.8) .. controls (7.5,0.3) and (7.7,0.1) .. (8, 0);
\coordinate [label=-180: $\cs_{k+1}$] (skk) at (8,1.1); 
\draw (skk);
\end{tikzpicture}
 \caption{\it Least energy branches $\cs_k,\cs_{k+1}$ plus one high energy branch bifurcating from $\la_k$}
\end{figure}
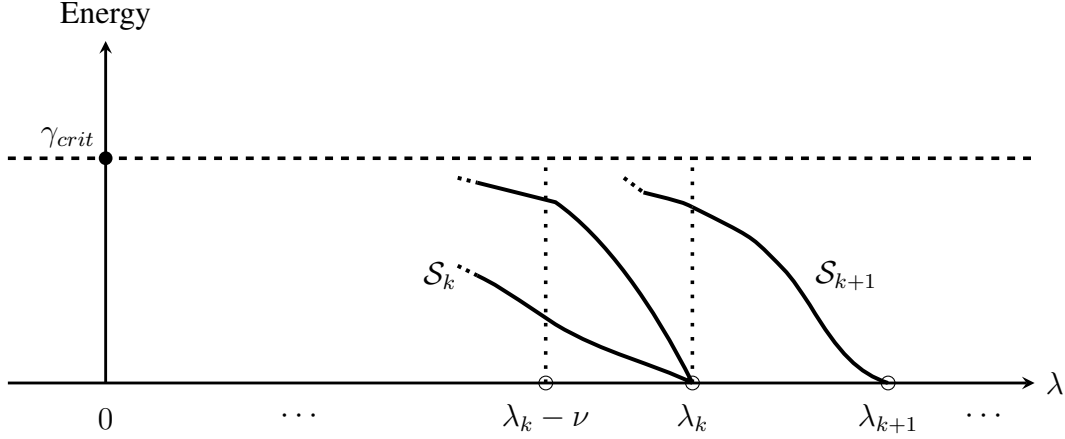

\section{Statement of the main results}
Let $(M,\ig,\si)$ be a compact spin manifold of dimension $m\geq2$.  The spectrum $\spec(D)=\{\la_k:k\in\Z\}$ consists only of eigenvalues with finite multiplicity which may be ordered as follows: \ $\ldots<\la_{-1}<\la_0\le0<\la_1<\la_2<\dots$;  \ $|\la_k|\to\infty$ as $|k|\to\infty$. We consider the following assumptions on the nonlinearity $f$; recall $F(s)=\int_0^s f(t)t\,dt$ .
\begin{itemize}
  \item[$(f_1)$] $f:[0,\infty)\to[0,\infty)$ is continuous and satisfies $f(0)=0$.
  \item[$(f_2)$] $\displaystyle \frac{f(s)}{s^\frac2{m-1}}\to0$
  as $s\to+\infty$.

  \item[$(f_3)$] The function $\displaystyle s\mapsto
  f(s)+s^\frac2{m-1}$ is strictly increasing.
  \item[$(f_4)$] $\displaystyle \frac{f(s)s}{F(s)^{\frac{m+1}{2m}}}\to0$ as $s\to+\infty$.
  \item[$(f_5)$] $\displaystyle \lim_{\vr\to0^+}\frac{\vr^{m-1}}{|\ln\vr|^{\max\{3-m,\,0\}}}
        \int_0^{\frac1\vr}F\left( \frac{\vr^{-\frac{m-1}2}}{(1+r^2)^{\frac{m-1}2}} \right) r^{m-1}dr=\infty$.
\end{itemize}

Now we can state our first main theorem.

\begin{Thm}\label{main thm}
\begin{itemize}
  \item[\rm a)] If $(f_1)$, $(f_2)$ and $(f_3)$ hold then \eqref{main equ} has a least energy solution $\psi_\la$ for every $\la>0$. The function
   \[
     \R^+ \to \R^+,\quad \la\mapsto \cl_\la(\psi_\la),
   \]
   is continuous and nonincreasing on each interval $[\la_k,\la_{k+1})$, $k\ge1$, and on $(0,\la_1)$. Moreover, $\psi_\la\to0$ hence $\cl_\la(\psi_\la)\searrow 0$ as $\la\nearrow\la_k$.

  \item[\rm b)] If $f\in \cc^1$, $f'(s)>0$ and $(f_1)$, $(f_4)$ and $(f_5)$ hold then \eqref{main equ} has a least energy solution for every $\la\in\R\setminus \{\lm_k:\, k\leq0\}$. The function
   \[
     \R\setminus\{\la_k:k\le0\} \to \R^+,\quad \la\mapsto \cl_\la(\psi_\la)
   \]
   is continuous and nonincreasing on each interval $[\la_{k-1},\la_k)$, if $k\ge2$, respectively on $(\la_{k-1},\la_k)$, if $k\le1$. Moreover, $\psi_\la\to0$ hence $\cl_\la(\psi_\la)\searrow 0$ as $\la\nearrow\la_k$.
\end{itemize}
\end{Thm}

\begin{Rem}\label{rem:nonlinearity}
%

  a) Assumptions $(f_1)$ and $(f_4)$ imply:
  \[
    \text{for every $\eps>0$ there exists $C_\eps>0$ such that $F(s) \le C_\eps+\eps s^{2^*}$ and $f(s) \le C_\eps+\eps s^{2^*-2}$.}
  \]
  This allows that $F(s)$ grows almost critically as $s\to\infty$. It also implies $(f_2)$, hence the assumptions of Theorem~\ref{main thm}~b) imply $(f_1)-(f_3)$. For later use we observe that $(f_1)-(f_3)$ imply that the functions $g(s):=f(s)+s^{\frac{2}{m-1}}=f(s)+s^{2^*-2}$ and $G(s):=\int_0^s g(t)t\,dt=F(s)+\frac1{2^*}s^{2^*}$ have the following properties:
\begin{itemize}
\item[(i)] $\R^+_0\to\R^+_0$, $s\mapsto g(s)s^2-2G(s)$ is strictly increasing.
\item[(ii)] For every $s_0>0$ there exists $c_0>0$ such that $g(s)s^2-2G(s)\geq c_0s^2$ for all $s\geq s_0$.
\end{itemize}

  b) The function $F(s) = \al s^p$ with $\al>0$ and $p\in(2,2^*)$ satisfies all conditions from Theorem~\ref{main thm}~b). The same is true for the function $F(s)=\frac{\al s^{2^*}}{\ln(1+s^q)}$ provided $\al>0$ and $q\in(0,\frac2{m-1}]$.

  c) It is a challenging open problem whether \eqref{main equ} has a least energy solution for $\la\le0$ in the situation of Theorem~\ref{main thm}~a), or for $\la=\la_k\le0$ in the situation of Theorem~\ref{main thm}~b). This may depend on an intricate combination of conditions on the geometry of $(M,g,\si)$ and properties of $f$.
\end{Rem}

As a consequence of Theorem~\ref{main thm} we obtain the following corollaries.

\begin{Cor}\label{cor1}
  The spinorial Brezis-Nirenberg equation \eqref{eq:BND} has a least energy solution for every $\la>0$.
\end{Cor}

This improves the result from \cite{Isobe:2011} who could only treat dimensions $m\ge4$ and required $\la\notin\spec(D)$.

\begin{Cor}\label{cor2}
  The equation
  \begin{equation}\label{eq:BNp}
    D\psi = \la\psi + \al|\psi|^{p-2}\psi + |\psi|^{2^*-2}\psi
  \end{equation}
  has a least energy solution for every
  $\la\in\R\setminus \{\lm_k:\, k\leq0\}$, $\al>0$, $2<p<2^*$.
\end{Cor}

Now we turn to the existence of high energy solutions.

\begin{Thm}\label{thm:mult}
  Suppose the hypotheses of Theorem 2.1 a) or b) hold. Then for $k\ge1$ there exists $a_k\in(\la_{k-1},\la_k)$ such that \eqref{main equ} has for $\la\in(a_k,\la_k)$ a solution $\psitil_\la$ so that the map
  \[
    (a_k,\la_{k+1})\to\R^+,\quad
    \la \mapsto \begin{cases} \cl_\la(\psitil_\la)&\text{if }\la<\la_k\\ \cl_\la(\psi_\la)&\text{if }\la\ge\la_k \end{cases}
  \]
  is continuous and non-increasing.
\end{Thm}

\begin{Rem}\label{rem:cont}
  By Theorem~\ref{thm:mult}, and its proof, the least energy solution $\psi_\la$ for $\la\in[\la_k,\la_{k+1})$ can be continued to $\la\in(a_k,\la_{k+1})$ in the sense that the energy, and the corresponding min-max description, changes continuously. We do not know whether the solutions $\psi_\la$ and $\psitil_\la$ depend continuously on $\la$. Since $\cl_\la(\psitil_\la)\to \cl_{\la_k}(\psi_{\la_k})>0$ and $\cl_\la(\psi_\la)\to0$ as $\la\nearrow\la_k$, we see that $\psitil_\la$ is different from $\psi_\la$. From a variational point of view $\psi_\la$ is a minimizer of $\cl_\la$ on the Nehari-Pankov manifold $\cp_\la$, $\la\in[a_k,\la_k)$, whereas $\psitil_\la$ minimizes $\cl_\la$ on a submanifold $\cq_{\la,k}$ of $\cp_\la$ of codimension $\dim\ker(D-\la_k)$, $\la\in(a_k,\la_k)$. $\cq_{\la,k}$ is a continuation of $\cp_{\la_k}$ for $\la<\la_k$ close to $\la_k$.
\end{Rem}

Next we state several multiplicity results. Observe that given a solution $\psi$ of \eqref{main equ} and $\zeta\in S^1\subset\C$ then $\zeta\psi$ is also a solution. Therefore we only count $S^1$-orbits of solutions in our multiplicity results. Clearly the solutions $\psi_\la$ and $\psitil_\la$ from Theorems~\ref{main thm} and \ref{thm:mult} lie on different $S^1$-orbits because they lie on different energy levels. The proof of the following local bifurcation theorem is standard and will not be given here.

\begin{Thm}\label{thm:bif}
  For $k\in\Z$ there exists $b_k<\la_k$ such that the following holds. For each $\la\in(b_k,\la_k)$ problem \eqref{main equ} has at least $d_k=\dim_\C\ker(D-\la_k)$ $S^1$-orbits of solutions $\psi_{\la,j}$, $j=1,\dots,d_k$. These satisfy $\psi_{\la,j}\to0$ as $\la\nearrow\la_k$.
\end{Thm}

This theorem follows immediately from \cite[Theorem~(3.1)]{Bartsch-Clapp:1990}. That the bifurcation is subcritical is a consequence of the sign of our nonlinearity. Theorem~\ref{thm:mult2} and its proof actually give a lower bound  $\la_k-b_k\ge\nu$ that is uniform in $k\in\Z$. 

Combining Theorems~\ref{thm:mult} and \ref{thm:bif} we obtain a multiplicity result for solutions of \eqref{main equ} and \eqref{eq:BND} provided $\la<\la_k$ is close to $\la_k$. If we order the $\psi_{\la,j}$ by $\cl_\la(\psi_{\la,j})\le\cl_\la(\psi_{\la,j+1})$ then the least energy solution $\psi_\la$ corresponds to $\psi_{\la,1}$. As a consequence of our results, for $\max\{a_k,b_k\}<\la<\la_k$ we have $d_k+1$ $S^1$-orbits of solutions, namely the $S^1$-orbits of the high energy solutions $\psitil_\la$ from Theorem~\ref{thm:mult}, and the $d_k$ $S^1$-orbits of low energy solutions from Theorem~\ref{thm:bif} which bifurcate from the trivial branch $\{(\la,0):\la\in\R\}$. 

Theorem~\ref{thm:mult} and its proof essentially yield that, for $k\ge1$ the least energy solution $\psi_\la=\psi_{\la,1}$ bifurcating from $\la_{k+1}$ can be continued a bit below $\la_{k}$, turning into the ``bound state'' $\psitil_\la$ for $\la<\la_{k-1}$. If $d_k>1$ then it is of course an interesting problem whether the bound states $\psi_{\la,j}$, $j=2,\dots,d_k$, bifurcating from $\la_{k+1}$ can also be extended to below $\la_{k}$, in the sense of Remark~\ref{rem:cont}. Our last result and its proof suggest that all solutions bifurcating from $\la_k$ can be extended to $\la\in(\la_k-\nu,\la_k)$ with
\begin{equation}\label{eq:nu}
  \nu:=\frac{m}2\bigg( \frac{\om_m}{\text{Vol}(M,\ig)} \bigg)^{\frac1m}>0.
\end{equation}
This holds for all $k\in\Z$.

\begin{Thm}\label{thm:mult2}
For $\la\in \R$ problem \eqref{main equ} has at least
\[
  \ell(\la)=\sum_{\la<\la_k<\la+\nu} \dim_\C\ker(D-\la_k)
\]
distinct $S^1$-orbits of solutions $\psi_{\la,j}$, $j=1,\dots,\ell(\la)$.
\end{Thm}

The Weyl formula and the symmetry of the eigenvalues of $D$ in dimension $m\not\equiv 3 \, (\text{mod }4)$ imply for $\La>0$ that 
\[
  d_+(\La) := \dim_\C\Bigg( \bigoplus_{0<\la_k\leq\La}\ker(D-\la_k) \Bigg) \sim c_M\La^m
\]
and
\[
  d_-(\La) := \dim\Bigg( \bigoplus_{-\La\le\la_k<0}\ker(D-\la_k) \Bigg) \sim c_M\La^m
\]
where $c_M>0$ is a constant depending on $m=\dim(M)$ and $\vol(M,\ig)$; see Proposition~\ref{Weyl} below. From this we immediately obtain the following result.

\begin{Prop}\label{prop:mult2}
If  $m\not\equiv 3$~(mod $4$) then $\ell(\la)\to\infty$ as $|\la|\to\infty$, hence the number of $S^1$-orbits of solutions of \eqref{main equ} becomes unbounded as $|\la|\to\infty$.
\end{Prop}

\begin{Rem}
a) If $m\equiv 3$~(mod $4$) the Schr\"odinger-Lichnerowicz formula only yields the following lower bound for the number $n(\la)$ of solutions of \eqref{main equ}:
\[
  n(\la)+n(-\la) \ge \sum_{\la_k\in(\la,\la+\nu)\cup(-\la-\nu,-\la)}d_k \to\infty \quad\text{as $\la\to\infty$.}
\]

b) We recall that Theorems~\ref{main thm} and \ref{thm:mult} do not yield any solution of \eqref{eq:BND} with $\la<0$ or for \eqref{main equ} with $\la=\la_k$, $k\le0$, whereas Theorems~\ref{thm:bif}, \ref{thm:mult2}, and Proposition~\ref{prop:mult2} do not distinguish between the sign of $\la$ or whether $\la$ lies in $\spec(D)$. Theorems~\ref{main thm}~b) and \ref{thm:mult2} yield the existence of a solution of \eqref{main equ} for every $\la\in\R$ except for a finite number of nonpositive eigenvalues. 
\end{Rem}

Finally we compare our theorems with some related results for the classical Brezis-Nirenberg problem
\begin{equation}\label{eq:BN}
\left\{
\begin{aligned}
-\De u &= \la u +|u|^{\frac4{N-2}}u && \text{in }\Om\\
u &= 0 && \text{on }\pa\Om
\end{aligned}
\right.
\end{equation}
on a smooth bounded domain $\Om\subset\R^N$.

\begin{Rem}
a) An interesting observation is that our results about \eqref{main equ} and \eqref{eq:BND} do not depend on the dimension unlike the classical result about \eqref{eq:BN} from \cite{BN1983} where dimension $N=3$ is special.

b) Theorem~\ref{thm:mult2} corresponds to the multiplicity result from \cite{Cerami-Fortunato-Struwe:1984} for \eqref{eq:BN}.

c) Other multiplicity results for sign-changing solutions of \eqref{eq:BN} have been proved in \cite{Clapp-Weth:2005, Devillanova-Solimini:2002}. These also depend on the dimension, in particular they differ for $4\le N\le6$ and $N\ge7$. It would be very interesting whether analogous results hold for \eqref{eq:BND}.
\end{Rem}

\section{Preliminaries on spinors}\label{Preliminaries}
\subsection{Spin structure and the Dirac operator}

Let $(M,\ig)$ be an $m$-dimensional Riemannian manifold with a chosen orientation. Let $P_{SO}(M)$ be the set of positively oriented orthonormal frames on $(M,\ig)$. This is a $SO(m)$-principal bundle over $M$. A {\it spin structure} on $M$ is a pair $\sa=(P_{Spin}(M),\vartheta)$ where $P_{Spin}(M)$ is a $Spin(m)$-principal bundle over $M$ and $\vartheta: P_{Spin}(M)\to P_{SO}(M)$ is a map such that the diagram
\begin{displaymath}
\xymatrix@R=0.3cm{
P_{Spin}(M)\times Spin(m)  \ar[r] \ar[dd]^{\displaystyle\vartheta\times \Theta}&  P_{Spin}(M)\ar[dd]^{\displaystyle\vartheta} \ar[dr] \\
&&  M \\
P_{SO}(M)\times SO(m) \ar[r] & P_{SO}(M)  \ar[ur] }
\end{displaymath}
commutes, where $\Theta: Spin(m)\to SO(m)$ is the nontrivial double covering of $SO(m)$. There is a topological condition for the existence of a spin structure, namely, the vanishing of the second Stiefel-Whitney class $\om_2(M)\in H^2(M,\Z_2)$. Furthermore, if a spin structure exists, it need not be unique. For these results we refer to \cite{Friedrich, Lawson}.

In order to introduce the spinor bundle, we recall that the Clifford algebra $Cl(\R^m)$ is the associative $\R$-algebra with unit, generated by $\R^m$ satisfying the relation $x\cdot y-y\cdot x=-2(x,y)$ for $x,y\in\R^m$ (here $(\cdot,\cdot)$ is the Euclidean scalar product on $\R^m$). It turns out that $Cl(\R^m)$ has a smallest representation $\rho: Spin(m)\subset Cl(\R^m)\to End(\mbs_m)$ of dimension $\dim_\C(\mbs_m)=2^{[\frac{m}2]}$ such that $\C l(\R^m):=Cl(\R^m)\otimes\C\cong End_\C(\mbs_m)$ as $\C$-algebra. The spinor bundle is then defined as the associated vector bundle
\[
  \mbs(M):=P_{Spin}(M)\times_\rho\mbs_m.
\]
Note that the spinor bundle carries a natural Clifford multiplication, a natural hermitian metric and a metric connection induced from the Levi-Civita connection on $TM$ (see \cite{Friedrich, Lawson}), this bundle satisfies the axioms of Dirac bundle in the sense that
\begin{itemize}
\item[$(i)$] for any $x\in M$, $X,Y\in T_xM$ and $\psi\in\mbs_x(M)$
\[
  X\cdot Y\cdot \psi + Y\cdot X\cdot\psi + 2\ig(X,Y)\psi = 0;
\]
\item[$(ii)$] for any $X\in T_xM$ and $\psi_1,\psi_2\in\mbs_x(M)$,
\[
  (X\cdot \psi_1,\psi_2)=-(\psi_1,X\cdot\psi_2),
\]
where $(\cdot,\cdot)$ is the hermitian metric on $\mbs(M)$;
\item[$(iii)$] for any $X,Y\in\Ga(TM)$ and $\psi\in\Ga(\mbs(M))$,
\[
  \nabla_X^\mbs(Y\cdot\psi)=(\nabla_XY)\cdot\psi+Y\cdot\nabla_X^\mbs\psi,
\]
where $\nabla^\mbs$ is the metric connection on $\mbs(M)$.
\end{itemize}
The Dirac operator is then defined on the spinor bundle $\mbs(M)$ as the composition
\begin{displaymath}
\xymatrixcolsep{1.6pc}\xymatrix{
 D:\Ga(\mbs(M)) \ar[r]^-{\nabla^\mbs}
  & \Ga(T^*M\otimes \mbs(M)) \ar[r]
  & \Ga(TM\otimes \mbs(M)) \ar[r]^-{\mathfrak{m}}
  & \Ga(\mbs(M))}
\end{displaymath}
where $\mathfrak{m}$ denotes the Clifford multiplication $\mathfrak{m}: X\otimes\psi\mapsto X\cdot\psi$.

\subsection{The Dirac spectrum and $H^{\frac12}$ spinors}

Let $\spec(D)$ denote the spectrum of the Dirac operator $D$. It is well-known that $D$ is essentially self-adjoint in $L^2(M,\mbs(M))$ and has compact resolvents (see \cite{Friedrich, Ginoux, Lawson}). Moreover, $\spec(D)=\{\la_k:k\in\Z\}$ is a closed subset of $\R$ consisting of a two-sided unbounded discrete sequence of eigenvalues with finite multiplicities. And the eigenspaces of $D$ form a complete orthonormal decomposition of $L^{2}(M,\mbs(M))$.

\begin{Prop}\label{Weyl}
If $m\not\equiv 3 \, (\text{mod }4)$, then the growth of the Dirac eigenvalues satisfies Weyl's asymptotic law:
\[
  \lim_{\La\to+\infty} \frac{d_\pm(\La)}{\La^m}= C_m  \mbox{\rm Vol} (M,\ig)
\]
where
\[
  d_+(\La)=\dim_\C\left( \bigoplus_{0<\la_k\leq\La}\ker(D-\la_k) \right),
  \quad
  d_-(\La)=\dim_\C\left( \bigoplus_{-\La\le\la_k<0}\ker(D-\la_k) \right)
\]
and $C_m>0$ is a dimensional constant.
\end{Prop}

\begin{proof}
For $\La>0$ let
\[
  N(\La)=\dim\Bigg( \bigoplus_{|\la_k|\le\La}\ker(D-\la_k) \Bigg).
\]
be the sum of the multiplicities of eigenvalues with modulus at most $\La$. Then the Schr\"odinger-Lichnerowicz formula for the Dirac operator yields
\[
  \lim_{\La\to+\infty}\frac{N(\La)}{\La^m}=c_m \mbox{\rm Vol} (M,\ig)
\]
for some positive dimensional constant $c_m$; see \cite[Corollary 2.43]{BGV}, for instance.

If the dimension $m\not\equiv 3 \, (\text{mod }4)$ then $\spec(D)$ is symmetric about the origin including the multiplicities \cite[Theorem~1.3.7]{Ginoux}, hence
\[
  d_+(\La)=d_-(\La)=\frac{N(\La)-\dim\ker(D)}2.
\]
The proposition follows with $C_m=c_m/2$.
\end{proof}

\begin{Rem}
In dimension $m\equiv 3\, (\text{mod }4)$ the spectrum $\spec(D)$ is not symmetric. In order to measure the lack of symmetry of $\spec(D)$ in this case, Atiyah, Patodi and Singer \cite{APS} introduced the so-called $\eta$-invariant of $D$, denoted by $\eta(D)$. In particular $\eta(D)=0$ as soon as $\spec(D)$ is symmetric. However, only few $\eta$-invariants are know explicitly. It was shown in \cite{Dahl} that, if $m\equiv 3\, (\text{mod }4)$, for arbitrary large $\La>0$, $j\in\N$ and $0<l_1<\dots<l_j<\La$ there exists a Riemannian metric $\ig$ such that $\spec(D_\ig)\cap(0,\La)=\{l_1<\dots<l_j\}$ and $\spec(D_\ig)\cap(-\La,0)=\emptyset$. The distribution of the spectrum in these dimensions becomes complicated. For further information about spectral theory for Dirac operators and for related topics we refer to the monograph by Ginoux \cite{Ginoux}.
\end{Rem}

We now define the operator $|D-\la|^{\frac12}: L^2(M,\mbs(M))\to L^2(M,\mbs(M))$ by
\[
  \psi = \sum_{k\in\Z} a_k\eta_k\ \  \mapsto\ \  |D-\la|^{\frac12}\psi := \sum_{k\in\Z}|\la_k-\la|^{\frac12}a_k\eta_k
\]
where $\eta_k\in\ker(D-\la_k)$ and $\int_M|\eta_k|^2d\vol_\ig=1$ for $k\in\Z$.
The Hilbert space
\[
  H^{\frac12}(M,\mbs(M))
   := \Big\{\psi \in L^2(M,\mbs(M)):\ |D|^{\frac12}\psi\in L^2(M,\mbs(M)) \Big\}.
\]
coincides with the Sobolev space $W^{\frac12,2}(M,\mbs(M))$ (see \cite{Adams, Ammann}). Let $P_\la^0:L^2(M,\mbs(M))\to\ker(D-\la)$ denote the projector. We can endow $H^{\frac12}(M,\mbs(M))$ with the inner product
\[
  \inp{\psi}{\va}_\la = \real\big(|D-\la|^{\frac12}\psi,|D-\la|^{\frac12}\va\big)_2+\real(P_\la^0\psi,P_\la^0\va)_2
\]
and the induced norm $\|\cdot\|_\la=\inp{\cdot}{\cdot}_\la^{\frac12}$, where $(\psi,\va)_2=\int_M(\psi,\va)d\vol_\ig$ is the $L^2$-inner product on spinors. Clearly $E := H^{\frac12}(M,\mbs(M))$ with this inner product has the orthogonal decomposition of three subspaces
\begin{\equ}\label{decomposition}
  E = E_\la^+\op E_\la^0\op E_\la^-
\end{\equ}
where $E_\la^\pm$ is the positive (resp.\ negative) eigenspace of $D-\la$, and $E_\la^0=\ker(D-\la)$ is its kernel (which may be trivial). The dual space of $E$ will be denoted by $E_\la^*=H^{-\frac12}(M,\mbs(M))$. By $P_\la^\pm:E\to E_\la^\pm$ and $P_\la^0:E\to E_\la^0$ we denote the orthogonal projections. If $\la$ is clear from the context, we simply write $\psi^\pm=P_\la^\pm(\psi)$ and $\psi^0=P_\la^0(\psi)$ for $\psi\in E$.

\section{The Palais-Smale condition}\label{sec:PS}
Equation \eqref{main equ} is the Euler-Lagrange equation of the functional
\[
  \cl_\la(\psi)
   = \frac12\int_M(D\psi,\psi)d\vol_\ig - \frac\la2\int_M|\psi|^2d\vol_\ig - \int_M F(|\psi|)d\vol_\ig - \frac1{2^*}\int_M|\psi|^{2^*}d\vol_\ig
\]
with $F(s)=\int_0^s f(t)t\,dt$. In the sequel, by $L^q$ we denote the Banach space $L^q(M,\mbs(M))$ for $q\geq1$ and by $|\cdot|_q$ we denote the usual $L^q$-norm.

The functional $\cl_\la$ is well defined on $E=H^{\frac12}(M,\mbs(M))$ and is of class $\cc^1$ as a consequence of the assumptions in Theorem~\ref{main thm} and since $E$ embeds into $L^q(M,\mbs(M))$ for $q\in[1,2^*]$. We write $\psi=\psi^++\psi^0+\psi^-\in E= E_\la^+\op E_\la^0\op E_\la^-$ according to the decomposition \eqref{decomposition}. Then $\cl_\la$ has the form
\[
  \cl_\la(\psi)
   = \frac12\big(\|\psi^+\|_\la^2-\|\psi^-\|_\la^2\big) - \int_M F(|\psi|)d\vol_\ig - \frac1{2^*}\int_M|\psi|^{2^*}d\vol_\ig.
\]
Obviously the functional $\ce$ defined in \eqref{ce lm} coincides with $\cl_\la$ when $f\equiv0$.

We now investigate the Palais-Smale condition for $\cl_\la$. Due to the non-compactness of the critical embedding $E=H^{\frac12}(M,\mbs(M))\hookrightarrow L^{2^*}(M,\mbs(M))$, one cannot expect that $\cl_\la$ and $\ce$ satisfy the Palais-Smale condition on $E$. We shall see that the Palais-Smale condition holds below a critical value.

\begin{Lem}\label{PS-bdd}
Suppose $f\in\cc[0,\infty)$ satisfies $(f_2)$.
Then a $(PS)_c$-sequence $(\psi_n)_n$ for $\cl_\la$ is bounded, for any $c\in\R$. Moreover, $c=0$ if and only if $\psi_n\to0$.
\end{Lem}

\begin{proof}
Assumption $(f_2)$ implies: for every $\eps>0$ there exists $C_\eps>0$ such that
\begin{\equ}\label{assumption f2}
\max\{f(s)s^2,\, F(s) \}\leq C_\eps+\eps s^{2^*}.
\end{\equ}

Now let $(\psi_n)_n$ be a $(PS)_c$-sequence, then
\begin{equation}\label{ps1}
\aligned
\cl_\la(\psi_n)=\frac12\big(\|\psi^+_n\|_\la^2-\|\psi^-_n\|_\la^2\big)
-\int_M F(|\psi_n|)d\vol_\ig  
-\frac1{2^*}\int_M|\psi_n|^{2^*}d\vol_\ig 
= c+o_n(1),
\endaligned
\end{equation}
\begin{equation}\label{ps2}
\aligned
\cl_\la'(\psi_n)[\psi_n]=\|\psi^+_n\|_\la^2-\|\psi^-_n\|_\la^2
-\int_M f(|\psi_n|)|\psi_n|^2d\vol_\ig 
-\int_M|\psi_n|^{2^*}d\vol_\ig 
=o(\|\psi_n\|_\la),
\endaligned
\end{equation}
and
\begin{equation}\label{ps3}
\aligned
\cl_\la'(\psi_n)[\psi_n^+-\psi_n^-]&=
\|\psi_n^++\psi_n^-\|_\la^2
-\real\int_M
f(|\psi_n|)(\psi_n,\psi_n^+-\psi_n^-)d\vol_\ig\\
&\qquad -\real\int_M
|\psi_n|^{2^*-2}(\psi_n,\psi_n^+-\psi_n^-)d\vol_\ig\\
&=o(\|\psi_n^++\psi_n^-\|_\la).
\endaligned
\end{equation}
Thus, we have for every $C>2c$ and $n$ large:
\begin{equation}\label{ps4}
\aligned
C+o(\|\psi_n\|_\la)&\geq 2\cl_\la(\psi_n)-\cl_\la'(\psi_n)[\psi_n]\\
&=\int_M \big(f(|\psi_n|)|\psi_n|^2-2F(|\psi_n|)\big)d\vol_\ig +
\frac1m |\psi_n|_{2^*}^{2^*}\\
&\geq\frac1{2m} |\psi_n|_{2^*}^{2^*}- C',
\endaligned
\end{equation}
where in the last inequality we have used \eqref{assumption f2}.
By \eqref{ps3}, we can deduce that
\[
\aligned
\|\psi_n^++\psi_n^-\|_\la^2\leq
\int_M\left( f(|\psi_n|)|\psi_n|
+|\psi_n|^{2^*-1}\right)|\psi_n^+-\psi_n^-|d\vol_\ig
+o(\|\psi_n^++\psi_n^-\|_\la).
\endaligned
\]
From this and the H\"older and Sobolev inequalities we get
\begin{equation}\label{ps4a}
\aligned
\|\psi_n^++\psi_n^-\|_\la^2
 &\leq C|\psi_n|_2^2+(C+1)|\psi_n|_{2^*}^{2^*-1}|\psi_n^+-\psi_n^-|_{2^*}+o(\|\psi_n^++\psi_n^-\|_\la) \\
 &\leq C'|\psi_n|_{2^*}^2+C'|\psi_n|_{2^*}^{2^*-1}\|\psi_n^++\psi_n^-\|_\la+o(\|\psi_n^++\psi_n^-\|_\la)
\endaligned
\end{equation}
where we have used $f(s)\leq C(1 + s^{\frac2{m-1}})$. Now, by taking into account \eqref{ps4}, we get
\[
\|\psi_n^++\psi_n^-\|_\la^2
\leq C\big(1+o(\|\psi_n\|_\la^{\frac2{2^*}})\big)+C\big(1+o(\|\psi_n\|_\la^{1-\frac1{2^*}})\big)\|\psi_n^++\psi_n^-\|_\la.
\]
Note that $\dim E_\la^0<\infty$, hence any two norms on $E_\la^0$ are equivalent. Therefore we have
\begin{equation}\label{ps4b}
\|\psi_n^0\|_\la^2 = |\psi_n^0|_2^2\leq|\psi_n|_2^2 \leq C|\psi_n|_{2^*}^2 \leq C\big(1+o(\|\psi_n\|_\la^{\frac2{2^*}})\big)
\end{equation}
which implies
\[
\|\psi_n\|_\la^2\leq C\big(1+o(\|\psi_n\|_\la^{\frac2{2^*}})\big)+
C\big(1+o(\|\psi_n\|_\la^{1-\frac1{2^*}})\big)
\|\psi_n\|_\la.
\]
Now the boundedness of $(\psi_n)_n$ follows from $2^*>2$.

If $c=0$ then \eqref{ps4} and the boundedness of $(\psi_n)_n$ imply $|\psi_n|_{2^*}\to0$, hence $\psi_n\to0$ by \eqref{ps4a} and \eqref{ps4b}.
\end{proof}

The next lemma has been proved in \cite{Isobe:2011}. Let
\[
  \ce_0(\psi) = \frac12\int_M(D\psi,\psi)d\vol_\ig - \frac1{2^*}\int_M|\psi|^{2^*}d\vol_\ig.
\]
be the functional associated to \eqref{eq:geom}.

\begin{Lem}\label{Isobe}
Let $(\psi_n)_n$ be a Palais-Smale sequence for $\ce_0$ such that $\psi_n\rightharpoonup 0$ in $E$ and
\[
  \liminf_{n\to\infty}\int_M|\psi_n|^{2^*}d\vol_\ig>0.
\]
Then
\[
  \liminf_{n\to\infty}\ce_0(\psi_n)
   = \liminf_{n\to\infty}\frac1{2m}|\psi_n|_{2^*}^{2^*}
   \geq \frac1{2m}\left(\frac{m}2\right)^m \om_m.
\]
\end{Lem}

As a consequence of Lemma~\ref{Isobe}, we obtain the following compactness result below the critical value
\begin{equation}\label{eq:crit-val}
  \ga_{crit} := \frac1{2m}\left(\frac{m}2\right)^m \om_m.
\end{equation}

\begin{Prop}\label{PS-condition}
Suppose $f\in\cc[0,\infty)$ satisfies $(f_2)$. Then the functional $\cl_\la$ satisfies the $(PS)_c$-condition for any $c<\ga_{crit}$.
\end{Prop}

\begin{proof}
Let $(\psi_n)_n$ be a $(PS)_c$-sequence for $\cl_\la$ on $E$. By Lemma \ref{PS-bdd}, $(\psi_n)_n$ is bounded and $\psi_n\to0$ if (and only if) $c=0$. Now suppose $c>0$ and, up to a subsequence if necessary, $\psi_n\rightharpoonup \psi_0$ in $E$. Then we have
\begin{\equ}\label{subcritical embedding}
\psi_n\to \psi_0 \quad \text{in } L^p \text{ for any }
1\leq p<2^*
\end{\equ}
and the limit spinor $\psi_0$ satisfies
\begin{\equ}\label{limit spinor equ}
D\psi_0=\lm\psi_0+
f(|\psi_0|)\psi_0+|\psi_0|^{2^*-2}\psi_0 \quad
\text{on } M.
\end{\equ}

We claim that $\psi_0\neq0$ in $E$. Indeed, assume to the contrary that $\psi_0=0$. Then it follows from \eqref{subcritical embedding} that $(\psi_n)_n$ is also a $(PS)$-sequence for $\ce_0$. Moreover,
\[
\liminf_{n\to\infty}\frac1{2m}\int_M|\psi_n|^{2^*}d\vol_\ig
=\liminf_{n\to\infty}\cl_\la(\psi_n)
-\frac12\cl_\la'(\psi_n)[\psi_n]
=c>0.
\]
Therefore, by Lemma \ref{Isobe} we have
\[
c=\liminf_{n\to\infty}\cl_\la(\psi_n)
=\liminf_{n\to\infty}\ce_0(\psi_n)
\geq\frac1{2m}\big(\frac{m}2\big)^m \om_m
\]
which contradicts the assumption $c<\frac1{2m}\big(\frac{m}2\big)^m \om_m$.

Now let us set $\va_n=\psi_n-\psi_0$. Then $\va_n\rightharpoonup 0$ in $E$. By the compact embedding $E\hookrightarrow L^p$ for $1\leq p<2^*$, we easily get
\begin{\equ}\label{dd1}
\int_M F(|\psi_n|)d\vol_\ig=\int_M F(|\psi_0|)d\vol_\ig
+o_n(1)
\end{\equ}
and, for arbitrarily $\psi\in E$ with $\|\psi\|_\la\leq1$,
\begin{\equ}\label{dd2}
\real\int_M f(|\psi_n|)(\psi_n,\psi)d\vol_\ig=
\real\int_M f(|\psi_0|)(\psi_0,\psi)d\vol_\ig
+o_n(1).
\end{\equ}
On the other hand, arguing exactly as in \cite[Lemma 5.2]{Isobe:2011}, we obtain  the Brezis-Lieb type result for the integrand of critical part, that is,
\begin{\equ}\label{dd3}
\int_M|\psi_n|^{2^*}d\vol_\ig = \int_M |\va_n|^{2^*}d\vol_\ig + \int_M |\psi_0|^{2^*}d\vol_\ig + o_n(1)
\end{\equ}
and, for arbitrarily $\psi\in E$ with $\|\psi\|_\la\leq1$,
\begin{\equ}\label{dd4}
\real\int_M|\psi_n|^{2^*-2}(\psi_n,\psi) =
\real\int_M |\va_n|^{2^*-2}(\va_n,\psi)  +
\real\int_M |\psi_0|^{2^*-2}(\psi_0,\psi)  + o_n(1).
\end{\equ}
Therefore, combining \eqref{dd1}-\eqref{dd4}, we infer that
\begin{\equ}\label{dd5}
\cl_\la(\psi_n)=\ce_0(\va_n)+\cl_\la(\psi_0)+o_n(1)
\end{\equ}
and
\begin{\equ}\label{dd6}
\cl_\la'(\psi_n)=\ce_0'(\va_n)+o_n(1),
\end{\equ}
where we have used \eqref{limit spinor equ}, i.e.\ $\cl_\la'(\psi_0)=0$, in the last equality.

\eqref{dd5} and \eqref{dd6} imply that $(\va_n)_n$ is a $(PS)$-sequence for $\ce_0$. If $|\va_n|_{2^*}\to0$ then we easily get $\va_n\to0$ in $E$ which gives the compactness of $(\psi_n)_n$ (cf.\ Lemma~\ref{PS-bdd}). If $\liminf_{n\to\infty}|\va_n|_{2^*}>0$, then it follows from Lemma \ref{Isobe} that
\[
\liminf_{n\to\infty}\ce_0(\va_n) \geq \frac1{2m}\big(\frac{m}2\big)^m \om_m.
\]
But this, together with \eqref{dd5}, implies $\liminf_{n\to\infty}\cl_\la(\psi_n) \geq \frac1{2m}\big(\frac{m}2\big)^m \om_m$ which contradicts our assumption. Hence we must have that $(\psi_n)_n$ is compact in $E$.
\end{proof}

\section{The min-max scheme}\label{sec:minmax}
The functional $\cl_\la\in\cc^1(E)$ has the form
\[
  \cl_\la(\psi) = \frac12\left(\|\psi^+\|_\la^2-\|\psi^-\|_\la^2\right) - \ck(\psi)
\]
with
\[
  \ck(\psi) = \int_M F(|\psi|)d\vol_\ig + \frac1{2^*}\int_M|\psi|^{2^*}d\vol_\ig.
\]

We need to investigate the properties of $\ck$. In the sequel we always assume $(f_1)-(f_3)$.

\begin{Lem}\label{lem1}
 $\ck(0)=0$ and \ $\frac12 \ck'(\psi)[\psi] > \ck(\psi) > 0$ \ for every $\psi\ne0$.
\end{Lem}

\begin{proof}
Using $(f_3)$ we obtain for $s>0$:
\[
F(s)+\frac1{2^*}s^{2^*}
 =\int_0^s f(t)t+t^{2^*-1}dt
  < \int_0^1 \big(f(s)+s^{2^*-2}\big) t\, dt \\
  =\frac12\big( f(s)s^2+s^{2^*}\big)
\]
The lemma follows immediately.
\end{proof}

\begin{Lem}\label{lem2}
  $\ck$ is strictly convex, hence weakly lower semi-continuous .
\end{Lem}

\begin{proof} Observe that $\ck(\psi)=\int_M G(|\psi|)d\vol_\ig$ and the function $G(s)=F(s)+\frac1{2^*}s^{2^*}$ is strictly convex as a consequence of $(f_3)$.
\end{proof}

Recall the decomposition $E= E_\la^+\op E_\la^0\op E_\la^-$ and let
\[
  S_\la^+ := \{\phi\in E_\la^+:\|\phi\|_\la=1\}.
\]
For $\phi\in S_\la^+$ we set
\[
  E_\la(\phi) := \R\phi\op E_\la^0\op E_\la^- \qquad\text{and}\qquad
  \ehat_\la(\phi) := \{t\phi+\chi: t\ge0,\ \chi\in E_\la^0\op E_\la^-\}\subset E_\la(\phi).
\]

\begin{Lem}\label{lem3}
  For each $\phi\in S_\la^+$ there exists a unique nontrivial critical point $\mu_\la(\phi) \in \ehat_\la(\phi)$ of the constrained functional $\cl_\la|_{\ehat_\la(\phi)}$. Moreover the following hold:
  \begin{itemize}
  \item[\rm a)] $\mu_\la(\phi)$ is the global maximum of $\cl_\la|_{\ehat_\la(\phi)}$.
  \item[\rm b)] $\mu_\la(\phi)^+$ is bounded away from $0$, i.e.\ $\|\mu_\la(\phi)^+\|_\la\ge\de$ for some $\de>0$.
  \item[\rm c)] $\mu_\la:S_\la^+\to \ehat_\la(\phi)$ is bounded on compact subsets of $S_\la^+$.
  \end{itemize}
\end{Lem}

\begin{proof}
We first prove that $\sup\cl_\la|_{\ehat_\la(\phi)} > 0$ is achieved. Observe that there exists a constant $C>0$ such that
\begin{\equ}\label{decomposition L2^*}
C\big( |t\phi|_{2^*}+|\chi^0|_{2^*}+|\chi^-|_{2^*}\big) \leq |t\phi+\chi|_{2^*}
\quad\text{for all $t\phi+\chi\in \ehat_\la(\phi)$}
\end{\equ}
because $\R\phi\op E_\la^0$ is of finite dimension. This and $F\ge0$ imply
\[
\cl_\la(t\phi+\chi)
 \leq \frac{t^2}2-\frac12\|\chi^-\|_\la^2 - C'\big( |t\phi|_{2^*}+|\chi^0|_{2^*}+|\chi^-|_{2^*}\big)^{2^*}
\]
on $\ehat_\la(\phi)$, hence $\cl_\la(\psi)\le 0$ if $\|\psi\|_\la$ is large. On the other hand, $\sup_{\ehat_\la(\phi)}\cl_\la\ge\al>0$ because $\cl_\la(t\phi)=\frac{t^2}2+o(t^2)$ as $t\to0$, uniformly in $\phi\in S_\la^+$. Here we used that for every $\eps>0$ there exists $C_\eps>0$ such that $F(s)\le\eps s^2+C_\eps s^{2^*}$. This together with the weak upper semi-continuity of $\cl_\la$ on $\ehat_\la(\phi)$ implies that the supremum is achieved at some $\mu_\la(\phi)$. It also follows that $\|\mu_\la(\phi)^+\|$ is bounded away from 0, and that $\mu_\la$ is bounded on compact sets.

Now we prove that any critical point $\psi\ne0$ of $\cl_\la|_{\ehat_\la(\phi)}$ is a strict global maximum of $\cl_\la|_{\ehat_\la(\phi)}$, hence it is unique. Let $\psi\in \ehat_\la(\phi)$ be a nontrivial critical point of $\cl_\la|_{\ehat_\la(\phi)}$ and $\va=(1+s)\psi+\chi\in\ehat_\la(\phi)\setminus\{\psi\}$ be an arbitrary element of $\ehat_\la(\phi)$; here $s\geq-1$ and $\chi\in E_\la^0\op E_\la^-$. We set $G(s)=F(s)+\frac1{2^*}s^{2^*}$, $g(s)=f(s)+s^{2^*-2}$ and fix $\chi$. Setting
\begin{\equ}\label{eq:def-h}
  h(s) := \frac{s^2+2s}{2}g(|\psi|)|\psi|^2 + (s+1)g(|\psi|)\Real (\psi,\chi) + G(|\psi|) - G(|(1+s)\psi+\chi|)
\end{\equ}
an elementary calculation shows:
\begin{\equ}\label{unique1}
\begin{aligned}
  \cl_\la(\va)-\cl_\la(\psi)
    &= \frac12\int_M\big((D\chi,\chi)-\la|\chi|^2\big)d\vol_\ig
          + \ck'(\psi)\big[ \frac{s^2+2s}2\psi+(1+s)\chi \big]\\
    &\hspace{1cm}
          +\ck(\psi)-\ck((1+s)\psi+\chi)\\
    &= -\frac12\|\chi^-\|_\la^2 + \int_M h(s)d\vol_\ig.
\end{aligned}
\end{\equ}
It is sufficient to prove that the integrand $h(s)(x)$ is negative if $\psi(x)\ne0$. Observe that $h(-1)=-\frac12g(|\psi|)|\psi|^2+G(|\psi|)-G(|\chi|)<0$ and $\displaystyle \lim_{s\to\infty}h(s)=-\infty$. We may assume that $h$ attains its maximum at some $s_*\in(-1,\infty)$, otherwise we are done. We distinguish $s$ and $s_*$ here and  emphasize that $s_*$ is a real function depending on $x$, and we can set $s_*(x)=-1$ provided that $h(-1)(x)$ is the extremum. We drop the argument $x\in M$ in the sequel. Then, denoted by $\va_*=(1+s_*)\psi+\chi$, we have:
\[
  0=h'(s_*)=\big(g(|\psi|)-g(|\va_*|)\big)\Real(\psi,\va_*).
\]
Now, if $\Real(\psi,\va_*)\ne0$ hypothesis $(f_3)$ implies $|\psi|=|\va_*|>0$, hence for $\va_*\ne\psi$:
\begin{\equ}\label{eq:est-h-1}
\aligned
  h(s) &\le h(s_*) = g(|\psi|)\Real\big(\psi, \frac{s_*^2+2s_*}2\psi+(1+s_*)\chi\big) \\
   &= -\frac{s_*^2}2g(|\psi|)|\psi|^2 - (1+s_*)g(|\psi|)\big(|\psi|^2-\Real(\psi,\va_*)\big)\\
   &= -\frac{1}2g(|\psi|)|\chi|^2 < 0
\endaligned
\end{\equ}
If $\Real(\psi,\va_*)=0$, i.e.\ $\Real(\psi,\chi) = -(1+s_*)|\psi|^2$ then
\begin{\equ}\label{eq:est-h-2}
\aligned
  h(s) &\le h(s_*) = -\frac{(s_*+1)^2}2g(|\psi|)|\psi|^2 - \frac12g(|\psi|)|\psi|^2 + G(|\psi|) - G(|\va_*|)  \\
        &< - \frac12g(|\psi|)|\psi|^2 + G(|\psi|)
         < 0.
\endaligned
\end{\equ}
\end{proof}

Now we consider the functional
\[
  \cm_\la: S_\la^+\to\R,\quad \cm_\la(\phi) := \cl_\la(\mu_\la(\phi)).
\]

\begin{Prop}\label{reduction}
  \begin{itemize}
  \item[\rm a)] $\cm_\la\in \cc^1(S_\la^+)$ and
      \[
        \cm_\la'(\phi)[\chi] = \|\mu_\la(\phi)^+\|_\la\,\cl_\la'(\mu_\la(\phi))[\chi] \quad \text{for all }\chi\in T_\phi(S_\la^+).
      \]
  \item[\rm b)] If $(\phi_n)_n$ is a Palais-Smale sequence for $\cm_\la$ then $(\mu_\la(\phi_n))_n$ is a Palais-Smale sequence for $\cl_\la$. If $(\psi_n)_n$ is a bounded Palais-Smale sequence for $\cl_\la$ then $\big(\frac{1}{\|\psi_n^+\|_\la}\psi_n^+\big)_n$ is a Palais-Smale sequence for $\cm_\la$.
  \item[\rm c)] $\phi\in S_\la^+$ is a critical point of $\cm_\la$ if and only if $\mu_\la(\phi)$ is a nontrivial critical point of $\cl_\la$. The corresponding critical values coincide.
  \end{itemize}
\end{Prop}

\begin{proof}
  By Lemmas~\ref{lem1}-\ref{lem3} we may apply \cite[Corollary~33]{Szulkin-Weth:2010}.
\end{proof}

\begin{Rem}
  The set $\cp_\la:=\{\mu_\la(\phi):\,\phi\in S_\la^+\}$ is the Nehari-Pankov manifold associated to $\cl_\la$. It is a topological manifold homeomorphic to $S_\la^+$ via the homeomorphism $\mu_\la:S_\la^+\to\cp_\la$. Neither $\mu_\la$ nor $\cp_\la$ need to be of class $\cc^1$ since $\ck$ is not $\cc^2$. It is therefore surprising that $\cm_\la=\cl_\la\circ \mu_\la$ is $\cc^1$. A general discussion of the construction and the properties of the Nehari-Pankov manifold $\cp_\la$ in an abstract setting can be found in \cite[Chapter~4]{Szulkin-Weth:2010}.
\end{Rem}

Theorem~\ref{main thm} follows if we can show that $\inf\cm_\la < \ga_{crit}$. This will be proved in section~\ref{subsec:proof la>0} for $\la>0$. For the proof of Theorem~\ref{thm:mult} we define for $k\in\Z$ and $\la_{k-1}<\la<\la_k$ close to $\la_k$ a submanifold  $\cq_{\la,k}\subset\cp_\la$ which has codimension $d_k=\dim E_{\la_k}^0$. We can then minimize $\cl_\la$ on $\cq_{\la,k}$ in order to obtain a second solution. For $\la=\la_k$  there holds $\cq_{\la_k,k}=\cp_{\la_k}$.

\begin{Lem}\label{continuation lemma1}
For $k\in\Z$ and $\sa>0$, there exists $\de_{k,\sa}>0$ such that for $\la\in(\la_k-\de_{k,\sa},\la_k]$ and $\phi\in S_{\la_k}^+$ with $|\phi|_2^2\geq\sa$ there exists a unique global maximum point $\nu_{\la,k}(\phi)\in\widehat E_{\la_k}(\phi)$ of the constrained functional $\cl_\la|_{\widehat E_{\la_k}(\phi)}$. Moreover, $P_{\la_k}^+(\nu_{\la,k}(\phi))$ is bounded away from $0$ uniformly in $\la\in(\la_k-\de_{k,\sa},\la_k]$ and $\phi\in S_{\la_k}^+$.
\end{Lem}

\begin{proof}
The existence of a global maximizer of $\cl_\la|_{\widehat E_{\la_k}(\phi)}$ is analogous to the one in the proof of Lemma~\ref{lem3}. We only need to show the uniqueness for $\la<\la_k$ close to $\la_k$. Let $\psi\in \widehat E_{\la_k}(\phi)$ be a global maximum of $\cl_\la|_{\widehat E_{\la_k}(\phi)}$, and observe that
\begin{\equ}\label{cl-bdd-below}
  \cl_\la(\psi)\geq  \inf\cm_{\lm_k}>0,
\end{\equ}
hence
\begin{\equ}\label{eq:int g>0}
  \int_M g(|\psi|)|\psi|^2d\vol_\ig = \|\psi^+\|_\la^2-\|\psi^-\|_\la^2 = 2\cl_\la(\psi) + 2\int_M G(|\psi|)d\vol_\ig \ge \al
\end{\equ}
 and
\begin{\equ}\label{g-G>0}
 \cl_\la(\psi)-\frac12\cl_\la'(\psi)[\psi] = \frac12\ck'(\psi)[\psi]-\ck(\psi)
     =\int_M\frac12 g(|\psi|)|\psi|^2-G(|\psi|)d\vol_\ig \geq \al.
\end{\equ}
for some $\al>0$.
These estimates hold uniformly in $\phi\in S_{\la_k}^+$, $\psi\in \widehat E_{\la_k}(\phi)$ a global maximum of $\cl_\la|_{\widehat E_{\la_k}(\phi)}$. Morover, \eqref{decomposition L2^*} implies
\[
\aligned
 \sup_{\chi\in E_{\lm_k}^0\op E_{\lm_k}^-}\cl_\la(\chi)
      &\leq  \max_{\eta\in S_\la^+\cap E_{\la_k}^0}\max_{t>0}
                   \left[ \frac{\la_k-\la}{2}t^2|\eta|_2^2-\frac{C}{2^*}t^{2^*}|\eta|_{2^*}^{2^*} \right]\\
    &= \max_{\eta\in S_\la^+\cap E_{\la_k}^0}\frac{(\la_k-\la)^m|\eta|_2^{2m}}{2m C^{m-1}|\eta|_{2^*}^{2^*(m-1)}}
      = O\big((\la_k-\la)^m\big).
\endaligned
\]
Hence, for $\la$ close to $\la_k$, we have $P_{\la_k}^+(\psi)=t\phi$ is bounded away from $0$ uniformly in $\la$ and $\phi\in S_{\la_k}^+$.

For $\phi\in S_{\la_k}^+$ with $|\phi|_2^2\geq\sa$ we consider an arbitrary element $\va=(1+s)\psi+\chi\in \widehat E_{\la_k}(\phi)\setminus\{\psi\}$, i.e.\ $s\geq-1$ and $\chi=\chi^0+\chi^-\in E_{\lm_k}^0\op E_{\lm_k}^-$ with $s\ne0$ or $\chi\ne0$. We fix $\chi$, define $h(s)$ as in \eqref{eq:def-h}, and deduce as in \eqref{unique1}:
\[
\aligned
  \cl_\la(\va)-\cl_\la(\psi)
    &= \frac12\int_M(D\chi,\chi)d\vol_\ig-\frac\lm2\int_M|\chi|^2d\vol_\ig + \int_M h(s)\,d\vol_\ig\\
    &= -\frac12\|\chi^-\|_{\la_k}^2 + \frac{\la_k-\la}{2}\int_M|\chi^0|^2d\vol_\ig + \int_M h(s)\,d\vol_\ig.
\endaligned
\]
Here the quadratic part is not negative definite. In order to prove $\cl_\la(\va)<\cl_\la(\psi)$ it is sufficient to consider $\va$ in a bounded region, i.e.~$\|\va\|_\la\le R$ for some constant $R>0$ independent of $\lm$, because $\cl_\la(\va)\to-\infty$ as $\|\va\|_\lm\to\infty$ uniformly for $\lm$ in a bounded interval. We also point out that $P_{\la_k}(\psi)=t\phi$, hence $|\psi|_2^2\geq \tau>0$ is bounded away from $0$ as mentioned above.

Observe that \eqref{eq:est-h-1} and \eqref{eq:est-h-2} imply
\[
  h(s) \le -\frac12\min\big\{g(|\psi|)|\chi|^2, \,g(|\psi|)|\psi|^2-2G(|\psi|)\big\}.
\]
Now we divide $M$ into two parts:
\[
  \Om_1=\{x\in M:\,  g(|\psi|)|\chi|^2\geq g(|\psi|)|\psi|^2-2G(|\psi|)\},
\]
and
\[
  \Om_2=\{x\in M:\,  g(|\psi|)|\chi|^2< g(|\psi|)|\psi|^2-2G(|\psi|)\}.
\]

\noindent
{\sc Case 1:} $\int_{\Om_1}|\psi|^2\,d\vol_\ig\geq \frac\tau2$.\\
In this case, setting $s_0:=\sqrt{\frac{\tau}{4\text{Vol}(M,\ig)}}$, Remark~\ref{rem:nonlinearity}~a) yields a constant $c_0>0$ such that
\[
  \int_{\Om_1}g(|\psi|)|\psi|^2-2G(|\psi|)\,d\vol_\ig \geq \int_{\{x\in\Om_1:\, |\psi|\geq s_0\}}c_0|\psi|^2 \ge \frac{c_0\tau}4.
\]
Then $\cl_\la(\va) \le \cl_\la(\psi)-\frac{c\tau}8$ for $\la_k-\la$ is small. Recall that we only need to consider $\va$, hence $\chi^0$ in a bounded region.

\medskip\noindent
{\sc Case 2:} $\int_{\Om_2}|\psi|^2d\vol_\ig\geq \frac\tau2$.\\
Here we observe that
\[
  \frac{\tau}2 \le \int_{\Om_2}|\psi|^2d\vol_\ig
   \le \left( \int_{\Om_2}|\psi|^{2^*-2}d\vol_\ig \right)^{1-\theta}\left( \int_{\Om_2}|\psi|^{2^*}d\vol_\ig \right)^{\theta}
\]
with $\theta=\frac{m-2}{m-1}$. Now the $H^{1/2}$-boundedness of $\psi$ and $g(s)=f(s)+s^{2^*-2}\ge s^{2^*-2}$ imply that
\begin{\equ}\label{eq:int g lower}
\int_{\Om_2}g(|\psi|)d\vol_\ig \geq \int_{\Om_2}|\psi|^{2^*-2}d\vol_\ig \ge \tau > 0
\end{\equ}
is bounded away from $0$. We now claim that the quadratic form
\[
  \chi\mapsto -\|\chi^-\|_{\la_k}^2-\int_{\Om_2} g(|\psi|)|\chi|^2d\vol_\ig
\]
is negative definite on $E_{\lm_k}^0\op E_{\lm_k}^-$, uniformly in $\phi$. Indeed, since $E_{\la_k}^0$ is finite dimensional and the nodal set of any eigenspinor in $E_{\lm_k}^0$ is of measure zero \cite{Bar1997}, it follows from \eqref{eq:int g lower} that there exists a positive constant $c_0>0$ such that
\[
c_0|\chi^0|_2^2\le \int_M g(|\psi|)|\chi^0|^2d\vol_\ig\le |\chi^0|_\infty^2\int_Mg(|\psi|)d\vol_\ig
\]
The constant $c_0$ can be chosen independently of $\psi$ because global maximizers of $\cl_\la|_{\widehat E_{\la_k}(\phi)}$ are bounded in $H^{1/2}$. Now we conclude:
\[
  \cl(\va)-\cl(\psi)
    \leq -\frac12\|\chi^-\|_{\la_k}^2 + \frac{\la_k-\la}{2}\int_M|\chi^0|^2d\vol_\ig - \int_{\Om_2} g(|\psi|)|\chi|^2d\vol_\ig
    <0
\]
provided $\lm_k-\lm$ is small enough.
\end{proof}

For $\sa>0$ and $\la_k-\de_{k,\si} < \la \le \la_k$, we have a map
\[
 \nu_{\la,k}=\nu_{\la,k,\si}: S_{\la_k,\si}^+=\big\{ \phi\in S_{\la_k}^+:\, |\phi|_2^2>\sa \big\} \to E.
\]
Observe that $\nu_{\la,k,\si}(\phi)=\nu_{\la,k,\si'}(\phi)$ for $|\phi|_2^2>\si>\si'$ which justifies the notation $\nu_{\la,k}$.

\begin{Rem}\label{rem:submanifold}
a) The map $\nu_{\la,k}: S_{\la_k,\si}^+\to E$ is an embedding, and its image $\cq_{\la,k,\si}:=\nu_{\la,k}(S_{\la_k,\si}^+)$ is a topological manifold homeomorphic to $S_{\la_k,\si}^+$. We claim that $\cq_{\la,k,\si}\subset\cp_\la$. Consider $\phi\in S_{\la_k,\si}^+$, and let $\nu_{\la,k}(\phi)=t\phi+\chi^0+\chi^-\in\widehat E_{\la,k}(\phi)\subset \R\phi\oplus E_{\la_k}^0\oplus E_{\la_k}^-$. This is a maximizer of $\cl_\la$ on $\widehat E_{\la_k}(\phi)$, hence on $\R^+(t\phi+\chi^0)+E_{\la_k}^- = \widehat E_{\la}\left(\frac{1}{\|t\phi+\chi^0\|_\la}(t\phi+\chi^0)\right)\subset\widehat E_{\la_k}(\phi)$. It follows that $\nu_{\la,k}(\phi)=\mu_\la\left(\frac{1}{\|t\phi+\chi^0\|_\la}(t\phi+\chi^0)\right)$.

b) The solution $\psitil_\la$ from Theorem~\ref{thm:mult} will be a minimizer of $\cl_\la$ on $\cq_{\la,k,\si}$. It is an interesting problem to determine the range of $\la<\la_k$ so that one can define $\cq_{\la,k,\si}$, for instance, whether it can be defined up to or below $\la_{k-1}$. This will depend on the geometry of $M$ and on the nonlinearity $f$.
\end{Rem}

For $\si>0$ and $\la\in(\la_k-\de_{k,\si},\la_k]$ we consider the functional
\[
  \cn_{\la,k}: S_{\lm_k,\si}^+\to\R, \quad \cn_{\lm,k}(\phi)=\cl_\la(\nu_{\la,k}(\phi)).
\]
We point out that $\nu_{\la_k,k}= \mu_{\la_k}$ and $\cn_{\la_k,k}= \cm_{\lm_k}$. The arguments from \cite[Corollary 33]{Szulkin-Weth:2010} imply the following proposition.

\begin{Prop}\label{small-reduction}
Fix $\si>0$ and $\la\in(\la_k-\de_{k,\sa},\la_k]$.
\begin{itemize}
\item[\rm a)] $\cn_{\la,k}\in \cc^1(S_{\lm_k,\si}^+)$ and
\[
  \cn_{\la,k}'(\phi)[\chi]=\big\|P_{\lm_k}^+\nu_{\la,k}(\phi)\big\|_{\la_k}\,\cl_\la'(\nu_{\la,k}(\phi))[\chi]
  \quad \text{for all } \chi\in T_\phi(S_{\lm_k,\si}^+).
\]
\item[\rm b)] If $(\phi_n)_n$ is a Palais-Smale sequence for $\cn_{\la,k}$ in $S_{\lm_k,\si}^+$ then $(\nu_{\la,k}(\phi_n))_n$ is a Palais-Smale sequence for $\cl_\la$.
\item[\rm c)] $\phi\in S_{\lm_k,\si}^+$ is a critical point of $\cn_{\la,k}$ if and only if $\nu_{\la,k}(\phi)$ is a nontrivial critical point of $\cl_\la$. The corresponding critical values coincide.
\end{itemize}
\end{Prop}

Now we show that $\inf\cn_{\la,k}>\inf\cm_{\la_k}$ is achieved for $\la\le\la_k$ close to $\la_k$, and is non-increasing and continuous in $\la$, provided $\inf \cm_{\la_k}<\ga_{crit}$. In section~\ref{subsec:proof la>0} we shall prove that $\inf \cm_{\la_k}<\ga_{crit}$ holds for all $k\ge1$. This proves Theorem~\ref{thm:mult}.

\begin{Lem}\label{lem:continuation}
\begin{itemize}
\item[\rm a)] For all $\phi\in S_{\lm_k}^+$ there holds $\nu_{\la,k}(\phi)\to\mu_{\lm_k}(\phi)$ as $\la\nearrow \la_k$. The convergence is uniform on compact subsets.
\item[\rm b)] If $\inf \cm_{\la_k}<\ga_{crit}$, then there exists $\sa>0$ such that
    \[
      \inf_{S_{\la_k,\si}}\cn_{\la,k}<\ga_{crit} \quad \text{for all } \la\in(\la_k-\de_{k,\sa},\la_k],
    \]
    and the infimum is achieved.
\item[\rm c)] The function $(\la_k-\de_{k,\si},\la_k] \to \R$, $\la\mapsto \inf\cn_{\la,k}$, is non-increasing and continuous.
\end{itemize}
\end{Lem}

\begin{proof}
a) We fix $\phi\in S_{\lm_k}^+$ and consider $\psi_\lm=\nu_{\la,k}(\phi)$. Observe that $\psi_\lm=t_\lm\phi+\chi_\lm$ with $t_\lm>0$ and $\chi_\lm\in E_{\lm_k}^0\op E_{\lm_k}^-$. By the boundedness of $\{\psi_\lm\}$, we have $\psi_\lm\rightharpoonup \psi_0$ with $t_\lm\to t_0$ and $\chi_\lm\rightharpoonup \chi_0$ as $\lm\nearrow \lm_k$. Then the weak lower semi-continuity of $|\cdot|_2$ and $\ck$ on $\widehat E_k(\phi)$ implies
\[
  \limsup_{\lm\nearrow\lm_k}\cn_{\la,k}(\phi)
   \leq \frac12\int_M(D\psi_0,\psi_0)d\vol_\ig - \frac{\lm_k}2\int_M|\psi_0|^2d\vol_\ig - \ck(\psi_0)
   \le \cm_{\la_k}(\phi).
\]
On the other hand, the monotonicity of $\cl_\la$ with respect to $\lm$ yields
\[
  0 \leq \cn_{\la,k}(\phi)-\cm_{\lm_k}(\phi)\leq \frac{\lm_k-\lm}{2}|\psi_\lm|_2^2\to0 \quad \text{as } \lm\nearrow\lm_k.
\]
It follows that $\psi_0=\mu_{\lm_k}(\phi)$ and
\[
\lim_{\lm\nearrow\lm_k}\big\|P_{\lm_k}^-\chi_\lm\big\|_{\lm}^2=\big\|P_{\lm_k}^-\chi_0\big\|_{\lm_k}^2.
\]
Since $E_{\lm_k}^0$ is of finite dimension, we have $\psi_\lm\to\mu_{\lm_k}(\phi)$ as
$\lm\nearrow\lm_k$.

b) Since $c_{\la_k}:=\inf \cm_{\la_k}<\ga_{crit}$, the infimum is achieved and the set of minimizers
\[
  A:=A_{\la_k} := \{\phi\in S_{\la_k}^+: \cm_{\la_k}(\phi)=c_{\la_k}\}
\]
is compact. Let $U_\rho(A)\subset S_{\la_k}^+$ be the $\rho$-neighborhood of $A$ with respect to $\|\cdot\|_{\la_k}$. We choose $\rho>0$ so that the $L^2$-norm is bounded away from $0$ on $U:=U_\rho(A)$, that is $U\subset S_{\la_k,\si}$ for some $\si>0$. Using again the Palais-Smale condition below $\ga_{crit}$ there exists $\al>0$ with $c_{\la_k}+3\al<\ga_{crit}$ and such that
\[
  \inf \cm_{\la_k}(S_{\la_k}^+\setminus U) \ge c_{\la_k}+2\al.
\]
The sets $B:=\mu_{\la_k}(A)$ and $V:=\mu_{\la_k}(U)$ are contained in the Nehari-Pankov manifold $\cp_{\la_k}$ from Remark~\ref{rem:submanifold}. By definition there holds
\[
  \inf\cl_{\la_k}\big(\cp_{\la_k}\setminus V\big) \ge c_{\la_k}+2\al\quad\text{and}\quad
   \inf\cl_{\la_k}(A) = c_{\la_k}.
\]
Using the monotonicity of $\cl_\la$ with respect to $\la$ we obtain for $\la<\la_k$ close to $\la_k$
\[
\inf\cl_{\la}\big( \cq_{\la,k,\sa}\setminus V_\la \big)\ge c_{\la_k}+2\al \quad \text{and} \quad
\inf\cl_{\la_k}(B_\la) \le c_{\la_k}+\al
\]
where $V_\la=\nu_{\la,k}(U)$ and $B_\la=\nu_{\la,k}(A)$. This implies
\[
\inf\cn_{\la,k}\big( S_{\la_k,\sa}^+\setminus U \big)\ge c_{\la_k}+2\al
\quad \text{and} \quad
\inf\cn_{\la,k}(A)\le c_{\la_k}+\al.
\]
Then it follows that $\cn_{\la,k}$ achieves its infimum on $S_{\la_k,\si}$ in the set $U$.

c) The function $\inf\cn_{\la,k}$ is non-increasing in $\la$ because $\cl_\la(\psi)$ is decreasing with respect to $\la$. For a given $\la\in(\la_k-\de_{k,\si},\la_k]$ there exists a spinor field $\phi\in S_{\la_k}^+$ such that
$\cn_{\la,k}(\phi)=\inf\cn_{\la,k}$. Then we have for $\la'<\la$:
\[
  \inf\cn_{\la,k} \le \inf\cn_{\la',k} \le \cn_{\la',k}(\phi) \to \cn_{\la,k}(\phi)=\inf\cn_{\la,k} \quad\text{as $\lm'\nearrow\lm$,}
\]
hence $\la\mapsto \inf\cn_{\la,k}$ is continuous from the left. In order to prove continuity from the right at $\la\in(\la_k-\de_{k,\si},\la_k)$, consider a sequence $\la_n'\searrow\la$ and let $\phi_n\in S_{\la_k}^+$ be a minimizer of $\cn_{\la_n',k}$. Then $\psi_n:=\nu_{\la'_n,k}(\phi_n)$ is a critical point of $\cl_{\la'_n}$ and $\cl_{\la_n'}(\psi_n)=\cn_{\la_n',k}(\phi_n)=\inf\cn_{\la_n',k}\le\inf\cn_{\la,k}$ , hence $(\psi_n)_n$ is a Palais-Smale sequence for $\cl_{\la}$ at a level $c=\lim_{n\to\infty}\inf\cn_{\la_n',k}\le\inf\cn_{\la,k}<\ga_{crit}$. It follows that $\psi_n\to\psi$ along a subsequence, and $\psi$ is a critical point of $\cl_\la$ at the level $c$. Equivalently, $\phi_n\to\phi \in S_{\la_k}^+$ with $\psi=\nu_{\la,k}(\phi)$, and $\phi$ is a critical point of $\cn_{\la,k}$. This implies $c\ge \inf\cn_{\la,k}$, and the continuity from the right follows.
\end{proof}

\section{Test spinors and auxiliary estimates}\label{Auxiliary estimates}
Our proof relies on the construction of a test spinor on $M$ in order to show $\inf\cm_\la < \ga_{crit}$ under the conditions of Theorem~\ref{main thm}. The test spinor comes from a spinor on $\R^m$ being cut-off and transplanted to $M$ so that it has support in a small neighborhood of an arbitrary point $p_0\in M$. We first need to recall a construction from the paper \cite{AGHM} of Ammann et al.

To begin with we fix $\psi_0\in\mbs_m$ with $|\psi_0| = m^{\frac{m-1}2}$ arbitrarily, set $\mu(x)=\frac1{1+|x|^2}$ for $x\in\R^m$ and define
\[
  \psi(x)=\mu(x)^{\frac{m}2}(1-x)\cdot\psi_0
\]
where. Then it is standard to verify that
\begin{\equ}\label{t1}
  D_{\R^m}\psi=\frac{m}2\mu\psi
\end{\equ}
and
\begin{\equ}\label{t2}
  |\psi| = m^{\frac{m-1}2}\mu^{\frac{m-1}{2}} = \left(\frac{m}{1+|x|^2}\right)^{\frac{m-1}{2}}.
\end{\equ}
We choose $\de<i(M)/2$ where $i(M)>0$ is the injectivity radius of $M$. Let $\eta:\R^m\to\R$ be a smooth cut-off function satisfying $\eta(x)=1$ if $|x|\le\de$ and $\eta(x)=0$ if $|x|\ge2\de$. Now we define $\psi_\vr:\R^m\to\mbs_m$ by
\begin{\equ}\label{test spinor}
  \va_\vr(x)=\eta(x)\psi_\vr(x) \quad \text{where}\quad \psi_\vr(x)=\vr^{-\frac{m-1}2}\psi(x/\vr).
\end{\equ}

In order to transplant the test spinor on $M$, we recall the Bourguignon-Gauduchon-trivialization. Here we fix $p_0\in M$ arbitrarily, and let $(x_1,\dots,x_m)$ be the normal coordinates given by the exponential map
\[
  \exp_{p_0}: \R^m\cong T_{p_0}M\supset U \to V\subset M,\quad x \mapsto p = \exp_{p_0}(x).
\]
For $p\in V$ let $G(p)=(\ig_{ij}(p))_{ij}$ denote the corresponding metric at $p$. Since $G(p)$ is symmetric and positive definite, the square root
$B(p)=(b_{ij}(p))_{ij}$ of $G(p)^{-1}$ is well defined, symmetric and positive definite. It can be thought of as linear isometry
\[
  B(p): (\R^m\cong T_{\exp_{p_0}^{-1}(p)}U,\ig_{\R^m}) \to (T_pV,\ig).
\]
We obtain an isomorphism of $SO(m)$-principal bundles:
\begin{displaymath}
\xymatrix{
  P_{SO}(U,\ig_{\R^m}) \ar[r]^{\displaystyle\phi}  \ar[d] & P_{SO}(V,\ig) \ar[d] \\
T_{p_0}M \supset U\ar[r]^{\ \ \displaystyle\exp_{p_0}} & V \subset M}
\end{displaymath}
where $\phi(y_1,\dots,y_m) = (By_1,\dots,By_m)$ for an oriented frame $(y_1,\dots,y_m)$ on $U$. Notice that $\phi$ commutes with the right action of $SO(m)$, hence it induces an isomorphism of spin structures:
\begin{displaymath}
\xymatrix{
Spin(m)\times U  =P_{Spin}(U,\ig_{\R^m}) \ar[r] \ar[d] & P_{Spin}(V,\ig) \subset P_{Spin}(M) \ar[d]\\
T_{p_0}M\supset U\ar[r]^{\ \ \displaystyle\exp_{p_0}} & V \subset M}
\end{displaymath}
Thus we obtain an isomorphism between the spinor bundles $\mbs(U)$ and $\mbs(V)$:
\begin{equation}\label{spin-iso}
  \mbs(U) := P_{Spin}(U,\ig_{\R^m})\times_\rho \mbs_m \longrightarrow \mbs(V) := P_{Spin}(V,\ig)\times_\rho \mbs_m \subset \mbs(M)
\end{equation}
where $(\rho,\mbs_m)$ is the complex spin representation.

Setting $e_i=B(\pa_i)=\sum_{j}b_{ij}\pa_j$ we obtain an orthonormal frame $(e_1,\dots, e_m)$ of $(TV,\ig)$. In order to simplify the notation, we use $\nabla$ and $\bar\nabla$, respectively, for the Levi-Civita connections on $(TU,\ig_{\R^m})$ and $(TV,\ig)$ and for the natural lifts of these connections to the spinor bundles $\mbs(U)$ and $\mbs(V)$, respectively. For the Clifford multiplications on these bundles, we shall write ``$\cdot$" in both cases, that is,
\[
  e_i\cdot\bar\psi=B(\pa_i)\cdot\bar\psi=\ov{\pa_i\cdot\psi}.
\]

Now a spinor $\va\in\Ga(\mbs(U))$ corresponds via the isomorphims \eqref{spin-iso} to a spinor $\bar\phi\in\Ga(\mbs(V))$. In particular, since the spinors $\va_\vr\in\Ga(\mbs(U))$ from \eqref{test spinor} have compact support in $U$ they correspond to spinors $\bar\va_\vr\in\Ga(\mbs(M))$ with compact support in $V$. These are not quite our test spinors, because they do not lie in $\cm$. In the rest of this subsection we prove some estimates that will be needed in order to control the critical values in the next section.

We write $D$ and $\bar D$ for the Dirac operators acting on $\Ga(\mbs(U))$ and $\Ga(\mbs(V))$, respectively. By \cite[Proposition 3.2]{AGHM} there holds
\begin{\equ}\label{cut-off spinor identity}
  \bar D \bar\va_\vr = \ov{D\va_\vr}+W\cdot\bar\va_\vr + X\cdot\bar\va_\vr + \sum_{i,j}(b_{ij}-\de_{ij})\ov{\pa_i\cdot\nabla_{\pa_j}\va_\vr}
\end{\equ}
with $W\in\Ga(Cl(TV))$ and $X\in\Ga(TV)$ given
\[
  W = \frac14\sum_{\substack{i,j,k \\ i\neq j\neq k\neq i}}\sum_{\al,\bt} b_{i\al}(\pa_{\al}b_{j\bt})b_{\bt k}^{-1}e_i\cdot e_j\cdot e_k,
\]
and
\[
  X = \frac14\sum_{i,k}\big( \bar\Ga_{ik}^i-\bar\Ga_{ii}^k\big)e_k = \frac12\sum_{i,k} \bar\Ga_{ik}^i e_k;
\]
here $(b_{ij}^{-1})_{ij}$ denotes the inverse matrix of $B$, and $\bar\Ga_{ij}^k:=\ig(\bar\nabla_{e_j}e_j,e_k)$. In the sequel we identify $x\in\R^m$ with $p=\exp_{p_0}x\in M$ for notational convenience. As remarked in \cite{AGHM, Isobe:2011}, observing that $B=(G^{-1})^{\frac12}$ and $G=I+O(|x|^2)$ as $|x|\to0$, we deduce
\begin{\equ}\label{WX}
  b_{ij}=\de_{ij}+O(|x|^2), \quad W=O(|x|^3) \quad \text{and} \quad X=O(|x|) \quad \text{as } |x|\to0.
\end{\equ}

In the sequel we use the notation $f_\vr\lesssim g_\vr$ for two functions $f_\vr$ and $g_\vr$, when there exists a constant $C>0$ independent of $\vr$ such that $f_\vr\leq C g_\vr$.

\begin{Lem}\label{derivative-esti}
Let $\bar\va_\vr\in\mbs(V)$ be as above and set $\bar R_\vr:=\bar D\bar\va_\vr-|\bar\va_\vr|^{2^*-2}\bar\va_\vr$. Then
\[
  \|\bar\va_\vr\|_{E_\la^*}
    \lesssim
      \begin{cases}
        \vr^{\frac12}  &\text{if } m=2,\\
        \vr|\ln\vr|^{\frac23} &\text{if } m=3,\\
        \vr   &\text{if } m\geq4,
      \end{cases}
\]
and
\[
  \|\bar R_\vr\|_{E_\la^*}
    \lesssim\begin{cases}
        \vr^{\frac{m-1}2} &\text{if } 2\leq m\leq 4, \\
        \vr^2|\ln\vr|^{\frac35} &\text{if } m=5, \\
        \vr^2 &\text{if } m\geq6.
      \end{cases}
\]
\end{Lem}

\begin{proof}
For $\psi\in E$ with $\|\psi\|\leq1$ and $\vr$ small we have
\begin{equation}\label{A0}
\begin{aligned}
\bigg|\int_M(\bar\va_\vr,\psi)d\vol_\ig\bigg|
 &= \bigg|\int_{B_{2\de}(p_0)}(\bar\va_\vr,\psi)d\vol_\ig\bigg|
  \leq \bigg(\int_{B_{2\de}(p_0)}|\bar\va_\vr|^{\frac{2m}{m+1}}d\vol_\ig\bigg)^{\frac{m+1}{2m}}|\psi|_{2^*}  \\
 &\lesssim \bigg(\int_{|x|\leq2\de}|\psi_\vr|^{\frac{2m}{m+1}}dx\bigg)^{\frac{m+1}{2m}}
  \lesssim \bigg(\vr^{\frac{2m}{m+1}}\int_0^{\frac{2\de}\vr}\frac{r^{m-1}}{(1+r^2)^{\frac{m(m-1)}{m+1}}}dr\bigg)^{\frac{m+1}{2m}}   \\
 &\lesssim  \begin{cases}
              \vr^{\frac12} &\text{if } m=2, \\
              \vr|\ln\vr|^{\frac23} &\text{if } m=3,\\
              \vr &\text{if } m\geq4,
            \end{cases}
\end{aligned}
\end{equation}
This implies the estimate on $\|\bar\va_\vr\|_{E_\la^*}$.

In order to estimate $\bar R_\vr$, \eqref{t1} and \eqref{t2} yield
\begin{\equ}\label{identity-D}
\aligned
  D\va_\vr &= \nabla\eta\cdot\psi_\vr+\eta D\psi_\vr\\
           &=\nabla\eta\cdot\psi_\vr+|\va_\vr|^{2^*-2}\va_\vr+(\eta-\eta^{2^*-1})|\psi_\vr|^{2^*-2}\psi_\vr.
\endaligned
\end{\equ}
Using \eqref{cut-off spinor identity}, we obtain
\[
  \bar R_\vr=A_1+A_2+A_3+A_4+A_5+A_6
\]
where
\begin{eqnarray*}
&&A_1=\ov{\nabla\eta\cdot\psi_\vr},\\
&&A_2=(\eta-\eta^{2^*-1})|\bar\psi_\vr|^{2^*-2}\bar\psi_\vr,\\
&&A_3= \eta W\cdot\bar\psi_\vr,\\
&&A_4=\eta X\cdot\bar\psi_\vr,\\
&&A_5=\eta \sum_{i,j} (b_{ij}-\de_{ij})\ov{\pa_i\cdot\nabla_{\pa_j}\psi_\vr},\\
&&A_6=\sum_{i,j}(b_{ij}-\de_{ij})\pa_j\eta\,\ov{\pa_i\cdot\psi_\vr}.
\end{eqnarray*}

In the following estimates we use that the support of $\eta$ is contained in $B_{2\de}(0)\subset\R^m$. Analogous to \eqref{A0}, using \eqref{t2} and \eqref{WX} we estimate:\\

$\displaystyle \begin{aligned}
\|A_1\|_{E_\la^*}
  &\lesssim&\left(\int_{B_{2\de}(p_0)}\big|\ov{\nabla\eta\cdot\psi_\vr}\big|^{\frac{2m}{m+1}}d\vol_\ig \right)^{\frac{m+1}{2m}}
   \lesssim&\left(\int_{\de\leq|x|\leq2\de}|\psi_\vr|^{\frac{2m}{m+1}}dx \right)^{\frac{m+1}{2m}}  \\
  &\lesssim&\left( \vr^{\frac{2m}{m+1}}\int_{\frac\de\vr}^{\frac{2\de}\vr}\frac{r^{m-1}}{(1+r^2)^\frac{m(m-1)}{m+1}}dr \right)^{\frac{m+1}{2m}}
   \lesssim \ \vr^{\frac{m-1}2}
\end{aligned}$\\

$\displaystyle \begin{aligned}
\|A_2\|_{E_\la^*}
 &\lesssim \left(\int_{B_{2\de}(p_0)}\big(\eta-\eta^{\frac{m+1}{m-1}}\big)^{\frac{2m}{m+1}}|\bar\psi_\vr|^{\frac{2m}{m-1}}d\vol_\ig\right)^{\frac{m+1}{2m}} \\
 &\lesssim \left(\int_{\de\leq|x|\leq2\de}|\psi_\vr|^{\frac{2m}{m-1}}dx \right)^{\frac{m+1}{2m}}
  \lesssim \left(\int_{\frac\de\vr}^{\frac{2\de}\vr}\frac{r^{m-1}}{(1+r^2)^m}dr \right)^{\frac{m+1}{2m}} \\
 &\lesssim \vr^{\frac{m+1}2}
\end{aligned}$\\

$\displaystyle \begin{aligned}
\|A_3\|_{E_\la^*}
 &\lesssim \left(\int_{B_{2\de}(p_0)}|W|^{\frac{2m}{m+1}}|\bar\psi_\vr|^{\frac{2m}{m+1}}d\vol_\ig \right)^{\frac{m+1}{2m}}
  \lesssim \left(\int_{|x|\leq2\de}|x|^{\frac{6m}{m+1}}|\psi_\vr|^{\frac{2m}{m+1}}dx \right)^{\frac{m+1}{2m}} \\
 &\lesssim \left( \vr^{\frac{8m}{m+1}}\int_0^{\frac{2\de}\vr}\frac{r^{\frac{6m}{m+1}+m-1}}{(1+r^2)^{\frac{m(m-1)}{m+1}}}dr \right)^{\frac{m+1}{2m}}
  \lesssim  \begin{cases}
              \vr^{\frac{m-1}2} &\text{if } 2\leq m\leq8\\
              \vr^4|\ln\vr|^{\frac59} &\text{if }  m=9\\
              \vr^4 &\text{if } m\geq10
            \end{cases}
\end{aligned}$\\

$\displaystyle \begin{aligned}
\|A_4\|_{E_\la^*}
 &\lesssim \left(\int_{B_{2\de}(p_0)}|X|^{\frac{2m}{m+1}}|\bar\psi_\vr|^{\frac{2m}{m+1}}d\vol_\ig \right)^{\frac{m+1}{2m}}
  \lesssim \left(\int_{|x|\leq2\de}|x|^{\frac{2m}{m+1}}|\psi_\vr|^{\frac{2m}{m+1}}dx \right)^{\frac{m+1}{2m}}  \\
 &\lesssim \left(\vr^{\frac{4m}{m+1}}\int_0^{\frac{2\de}\vr}\frac{r^{\frac{2m}{m+1}+m-1}}{(1+r^2)^{\frac{m(m-1)}{m+1}}}dr \right)^{\frac{m+1}{2m}}
  \lesssim  \begin{cases}
              \vr^{\frac{m-1}2} &\text{if } 2\leq m\leq4\\
              \vr^2|\ln\vr|^{\frac35} &\text{if } m=5\\
              \vr^2 &\text{if } m\geq6
            \end{cases}
\end{aligned}$\\

$\displaystyle \begin{aligned}
\|A_5\|_{E_\la^*}
 &\lesssim \left(\int_{|x|\leq2\de}|x|^{\frac{4m}{m+1}}|\nabla\psi_\vr|^{\frac{2m}{m+1}}dx \right)^{\frac{m+1}{2m}}
  \lesssim \left(\vr^{\frac{4m}{m+1}}\int_0^{\frac{2\de}\vr}\frac{r^{\frac{4m}{m+1}+m-1}}{(1+r^2)^{\frac{m^2}{m+1}}}dr \right)^{\frac{m+1}{2m}} \\
 &\leq \left(\vr^{\frac{4m}{m+1}}\int_0^{\frac{2\de}\vr}\frac{r^{\frac{2m}{m+1}+m-1}}{(1+r^2)^{\frac{m(m-1)}{m+1}}}dr \right)^{\frac{m+1}{2m}}
  \lesssim  \begin{cases}
              \vr^{\frac{m-1}2}  &\text{if } 2\leq m\leq4\\
              \vr^2|\ln\vr|^{\frac35} &\text{if } m=5\\
              \vr^2 &\text{if } m\geq6
            \end{cases}
\end{aligned}$\\

\noindent Here we used the inequality $|\nabla\psi(x)|\lesssim \mu(x)^{\frac{m}2}$ and the same estimate as for $\|A_4\|_{E_\la^*}$. Finally there holds:\\

$\displaystyle
\|A_6\|_{E_\la^*}
 \lesssim \left(\int_{|x|\leq2\de}|x|^{\frac{6m}{m+1}}|\psi_\vr|^{\frac{2m}{m+1}}dx \right)^{\frac{m+1}{2m}}
  \lesssim  \begin{cases}
              \vr^{\frac{m-1}2} &\text{if } 2\leq m\leq8\\
              \vr^4|\ln\vr|^{\frac59} &\text{if } m=9\\
              \vr^4 &\text{if } m\geq10
            \end{cases}
$\\

\noindent Here we used $|\nabla\eta(x)|\lesssim |x|$ and the same estimate as for $\|A_3\|_{E_\la^*}$.

From these estimates we finally obtain:
\[
\|\bar R_\vr\|_{E_\la^*}
 \lesssim  \begin{cases}
             \vr^{\frac{m-1}2} &\text{if } 2\leq m\leq 4 \\
             \vr^2|\ln\vr|^{\frac35} &\text{if } m=5 \\
             \vr^2 &\text{if } m\geq6
           \end{cases}
\]
\end{proof}

\begin{Lem}\label{energy-esti}
Let $\bar\va_\vr\in\mbs(V)$ be as above and let $\om_m$ stand for the volume of the standard sphere $S^m$. Then
\begin{\equ}\label{el2}
\int_M|\bar\va_\vr|^2d\vol_\ig
 \gtrsim  \begin{cases}
            \vr|\ln\vr|+O(\vr) &\text{if } m=2\\
            \vr + O(\vr^2) &\text{if } m\geq3
          \end{cases}
\end{\equ}
and
\begin{\equ}\label{ef}
\frac12\int_{M}(\bar D\bar\va_\vr,\bar\va_\vr)d\vol_\ig - \frac1{2^*}\int_M|\bar\va_\vr|^{2^*}d \vol_\ig
 \leq \frac1{2m}\left(\frac{m}2\right)^m\om_m + \begin{cases}
                                                  O(\vr)  &\text{if } m=2, \\
                                                  O(\vr^2|\ln\vr|) &\text{if } m=3,\\
                                                  O(\vr^2) &\text{if } m\geq4,
                                                \end{cases}
\end{\equ}
\end{Lem}

\begin{proof}
For the first estimate, by taking into account $d\vol_\ig=d\vol_{\R^m}+O(|x|^2)$ in normal coordinates at $p_0$ we have
\begin{equation*}
\begin{aligned}
\int_M|\bar\va_\vr|^2d\vol_\ig
 &=\int_{B_{2\de}(p_0)}|\bar\va_\vr|^2d\vol_\ig  \\
 &=\int_{|x|\leq\de}|\psi_\vr|^2dx+O\left(\int_{\de<|x|\leq2\de} |\psi_\vr|^2dx \right)+O\left(\int_{|x|\leq2\de}|x|^2|\psi_\vr|^2dx \right) \\
 &=\vr m^{m-1}\om_{m-1}\int_0^{\frac\de\vr}\frac{r^{m-1}}{(1+r^2)^{m-1}}dr
    + O\left( \vr\int_{\frac\de\vr}^{\frac{2\de}\vr}\frac{r^{m-1}}{(1+r^2)^{m-1}}dr \right) \\
 &\hspace{1cm} + O\left( \vr^3\int_0^{\frac{2\de}\vr}\frac{r^{m+1}}{(1+r^2)^{m-1}}dr \right)  \\
 &= A(\vr)+O(\vr^{m-1}) + \begin{cases}
                            O(\vr^{m-1}) &\text{if } m=2,3,\\
                            O(\vr^3|\ln\vr|) &\text{if } m=4,\\
                            O(\vr^3) &\text{if } m\geq5
                          \end{cases} \\
 &= A(\vr) + \begin{cases}
               O(\vr^{m-1}) &\text{if } m=2,3,\\
               O(\vr^3|\ln\vr|) &\text{if } m=4,\\
               O(\vr^3) &\text{if } m\geq5,
             \end{cases}
\end{aligned}
\end{equation*}
where
\[
A(\vr) = \begin{cases}
           \vr\big( \ln(\vr^2+\de^2)-2\ln\vr \big)m^{m-1}\om_{m-1} &\text{if } m=2  \\
           \vr m^{m-1}\om_{m-1}\int_0^\infty\frac{r^{m-1}}{(1+r^2)^{m-1}}dr &\text{if } m\geq3
         \end{cases}
\]
This implies \eqref{el2}.

Now we come to \eqref{ef}. Analogously to the arguments in Lemma~\ref{derivative-esti}, we shall use \eqref{cut-off spinor identity} and
\eqref{identity-D} in order to get
\[
  \int_M(\bar D\bar\va_\vr,\bar\va_\vr)d\vol_\ig = J_1+J_2+\dots+J_7
\]
where
\begin{eqnarray*}
&&J_1 = \real\int_M \eta\cdot(\ov{\nabla\eta\cdot\psi_\vr},\bar\psi_\vr)d\vol_\ig \\
&&J_2 = \int_M|\bar\va_\vr|^{2^*}d\vol_\ig \\
&&J_3 = \int_M(\eta-\eta^{2^*-1})\cdot\eta\cdot|\bar\psi_\vr|^{2^*}d\vol_\ig \\
&&J_4 = \real\int_M\eta^2\cdot(W\cdot\bar\psi_\vr,\bar\psi_\vr)d\vol_\ig \\
&&J_5 = \real\int_M\eta^2\cdot(X\cdot\bar\psi_\vr,\bar\psi_\vr)d\vol_\ig \\
&&J_6 = \real\sum_{i,j}\int_M\eta^2(b_{ij}-\de_{ij})(\ov{\pa_i\cdot\nabla_{\pa_j}\psi_\vr},\bar\psi_\vr)d\vol_\ig \\
&&J_7 = \real\sum_{i,j}\int_M\eta\cdot(b_{ij}-\de_{ij})\pa_j\eta\cdot(\ov{\pa_i\cdot\psi_\vr},\bar\psi_\vr)d\vol_\ig.
\end{eqnarray*}
We have:

$J_1=0$\\

$\displaystyle \begin{aligned}
J_2 &=\int_{|x|\leq\de}|\psi_\vr|^{2^*}dx+O\left(\int_{\de<|x|\leq2\de}|\psi_\vr|^{2^*}dx \right)
       + O\left(\int_{|x|\leq2\de}|x|^2|\psi_\vr|^{2^*}dx \right) \\
 &= m^m\om_{m-1}\int_0^{\frac\de\vr}\frac{r^{m-1}}{(1+r^2)^m}dr
       + O\left(\int_{\frac\de\vr}^{\frac{2\de}\vr}\frac{r^{m-1}}{(1+r^2)^m}dr \right) \\
 &\hspace{1cm} + O\left(\vr^2\int_0^{\frac{2\de}\vr}\frac{r^{m+1}}{(1+r^2)^m}dr \right) \\
 &= m^m\om_{m-1}\int_0^{\infty}\frac{r^{m-1}}{(1+r^2)^m}dr + O(\vr^m)
    + \begin{cases}
        O(\vr^2|\ln\vr|), &\text{if } m=2 \\
        O(\vr^2), &\text{if } m\geq3
      \end{cases} \\
 &= m^m\om_{m-1}\int_0^{\infty}\frac{r^{m-1}}{(1+r^2)^m}dr
    + \begin{cases}
        O(\vr^2|\ln\vr|) &\text{if } m=2\\
        O(\vr^2) &\text{if } m\geq3
      \end{cases}
\end{aligned}$\\

$J_3\lesssim \int_{\de\leq|x|\leq2\de}|\psi_\vr|^{2^*}dx \lesssim \vr^m$\\

$\displaystyle 
J_4 \lesssim \int_{|x|\leq2\de}|x|^3|\psi_\vr|^2dx \lesssim \vr^4\int_0^{\frac{2\de}\vr}\frac{r^{m+2}}{(1+r^2)^{m-1}}dr
 \lesssim \begin{cases}
             \vr^{m-1} &\text{if } 2\leq m\leq 4  \\
             \vr^4|\ln\vr| &\text{if } m=5\\
             \vr^4 &\text{if } m\geq6
           \end{cases}
$\\

$J_5=0$\\

$\displaystyle \begin{aligned}
J_6 &\lesssim \int_{|x|\leq2\de}|x|^2|\nabla\psi_\vr|\cdot|\psi_\vr|dx
 \lesssim \vr^2\int_0^{\frac{2\de}\vr}\frac{r^{m+1}}{(1+r^2)^{m-\frac12}}dr
 &\lesssim \begin{cases}
             \vr &\text{if } m=2\\
             \vr^2|\ln\vr| &\text{if } m=3\\
             \vr^2 &\text{if } m\geq4
           \end{cases}
\end{aligned}$\\

\noindent Here we used the inequality $|\nabla\psi(x)|\lesssim\mu(x)^{\frac{m}2}$.\\

$\displaystyle
J_7 \lesssim \int_{|x|\leq2\de}|x|^3|\psi_\vr|^2dx
 \lesssim \begin{cases}
            \vr^{m-1} &\text{if } 2\leq m\leq 4  \\
            \vr^4|\ln\vr| &\text{if } m=5\\
            \vr^4 &\text{if } m\geq6
          \end{cases}
$\\

\noindent Here we used the inequality $|\nabla\eta(x)|\lesssim x$ and the same estimate as for $J_4$.

Combining these estimates we deduce that
\[
\begin{aligned}
&\frac12\int_{M}(\bar D\bar\va_\vr,\bar\va_\vr)d\vol_\ig - \frac1{2^*}\int_M|\bar\va_\vr|^{2^*}d \vol_\ig\\
&\hspace{1cm}
 \leq \frac12 m^{m-1}\om_{m-1}\int_0^\infty\frac{r^{m-1}}{(1+r^2)^m}dr
   + \begin{cases}
       O(\vr) &\text{if } m=2 \\
       O(\vr^2|\ln\vr|) &\text{if } m=3\\
       O(\vr^2) &\text{if } m\geq4
     \end{cases}
\end{aligned}
\]
Finally \eqref{ef} follows upon taking into account that
\[
  \om_m = \int_{\R^m}\left(\frac2{1+|x|^2}\right)^mdx = 2^m\om_{m-1}\int_0^\infty\frac{r^{m-1}}{(1+r^2)^m}dr.
\]
\end{proof}

\section{Proof of the main results}\label{PMR}
As a consequence of the results from sections \ref{sec:PS} and \ref{sec:minmax} we need to prove
\begin{\equ}\label{target}
  \inf\cm_\la < \ga_{crit},
\end{\equ}
where $\cm:S^+\to\R$ is defined in Proposition~\ref{reduction}. The strategy is to find suitable modifications of the test spinors that lie on $\cm$, and to control the energy of these modifications.

\subsection{Proof of Theorem \ref{main thm} for $\lm>0$ and of Theorem~\ref{thm:mult}}\label{subsec:proof la>0}
Since $F\ge0$ by $(f_1)$, we have
\[
  \cl_\la(\psi)\leq \ce_\la(\psi) = \frac12\big(\|\psi^+\|_\la^2-\|\psi^-\|_\la^2\big) - \frac1{2^*}\int_M|\psi|^{2^*}d\vol_\ig
\]
for all $\psi\in E$. Let
\[
  T_\la:L^{2^*}=L^{2^*}(M,\mbs(M))\to E_\la^0=\ker(D-\lm), \quad T_\la(\psi) = \arg\min_{\phi\in E_\la^0} |\psi-\phi|_{2^*}^{2^*},
\]
i.e.\ $T_\la(\psi)\in E_\la^0$ is the best approximation of $\psi\in L^{2^*}$ in $E_\la^0$. This exists and is unique because the $L^{2^*}$ norm is uniformly convex and $E_\la^0$ is finite-dimensional. Of course, $T$ is not linear in general, and $T\equiv0$ when $E_\la^0=\{0\}$. Clearly we have $|\psi|_{2^*}^{2^*}\ge|\psi-T_\la(\psi)|_{2^*}^{2^*}$ and therefore
\begin{equation}\label{def:tilde-e}
  \ce_\la(\psi) \le \tce_\la(\psi) := \ce_\la(\psi-T_\la(\psi))
    = \frac12\big(\|\psi^+\|_\la^2-\|\psi^-\|_\la^2\big) - \frac1{2^*}\int_M|\psi-T_\la(\psi)|^{2^*}d\vol_\ig.
\end{equation}
This implies
\begin{\equ}\label{energy-estimate-M}
  \inf\cm_\la = \inf_{\phi\in S^+}\max_{\psi\in\ehat_\la(\phi)}\cl_\la(\psi)
            \leq \inf_{\phi\in S_\la^+}\max_{\psi\in\ehat_\la(\phi)}\tce_\la(\psi).
\end{\equ}

We need to collect some properties of $T_\la$ and of the functionals
\[
  \cf_\la:L^{2^*}\to\R,\quad \cf_\la(\psi)=\frac1{2^*}|\psi-T_\la(\psi)|_{2^*}^{2^*}
\]
and
\[
  \tce_\la:E\to\R,\quad \tce_\la(\psi) = \ce_\la(\psi-T_\la(\psi)) = \frac12(\|\psi^+\|_\la^2-\|\psi^-\|_\la^2)-\cf_\la(\psi).
\]

\begin{Lem}\label{lem:prop-T}
\begin{itemize}
  \item[\rm a)] For $\psi\in L^{2^*}$, $\phi\in E_\la^0$ and $t\in\R$ there holds: $T_\la(t\psi)=tT_\la(\psi)$, $T_\la(\psi+\phi)=T_\la(\psi)+\phi$ and $\cf_\la(\psi+\phi)=\cf_\la(\psi)$.
  \item[\rm b)] $T_\la$ is of class $\cc^1$ on $L^{2^*}\setminus E_\la^0$. Moreover, $T'_\la(\psi)[\psi]=T_\la(\psi)$ for all $\psi\in L^{2^*}\setminus E_\la^0$.
  \item[\rm c)] $\cf_\la$ is of class $\cc^2$ on $L^{2^*}\setminus E_\la^0$ with derivative
  \[
    \cf_\la'(\psi)[\phi] = \real\int_M|\psi-T_\la(\psi)|^{2^*-2}(\psi-T_\la(\psi),\phi)d\vol_\ig.
  \]
  \item[\rm d)] $\cf_\la$ is convex.
	\item[\rm e)] If $\psi\in\te=E_\la^+\oplus E_\la^-$ is a critical point of $\tce_\la$ then $\psi-T_\la(\psi)\in E$ is a critical point of $\ce_\la$.
\end{itemize}
\end{Lem}

\begin{proof}
a) is trivial. For the proof of b) we fix $\psi\in L^{2^*}\setminus E_\la^0$ and consider the map $\cf_{\la,\psi}:E_\la^0\to\R$ defined by $\cf_{\la,\psi}(\phi):=\cf_\la(\psi+\phi)=\frac1{2^*}|\psi-\phi|_{2^*}^{2^*}$. Observe that $\phi=T_\la(\psi)$ is the unique solution of $\cf'_{\la,\psi}(\phi)=0$ because $\cf_{\la\psi}$ is of class $\cc^2$ and strictly convex. A simple computation yields for $\phi\in E_\la^0$:
\begin{\equ}\label{msf-sec-derivative}
  \cf_{\la,\psi}''(T_\la(\psi))[\phi,\phi]\geq\int_M|\psi-T_\la(\psi)|^{2^*-2}\cdot|\phi|^2d\vol_\ig.
\end{\equ}
Since the nodal set of any $\phi\in E_\la^0\setminus\{0\}$ is of measure zero by \cite{Bar1997}, $\cf_{\la,\psi}''(T_\la(\psi))$ is positive definite. Now the differentiability of $T_\la$ follows from the implicit function theorem, and $T_\la'(\psi)[\psi]=T_\la(\psi)$ follows from a).

In order to prove c) observe that $\cf_\la'(\psi)$ is trivial on $E_\la^0$ by a). This implies for $\phi\in L^{2^*}$:
\[
\aligned
  \cf_\la'(\psi)[\phi] &= \real\int_M|\psi-T_\la(\psi)|^{2^*-2}\big(\psi-T_\la(\psi),\phi-T_\la'(\psi)[\va]\big)d\vol_\ig\\
   &= \real\int_M|\psi-T_\la(\psi)|^{2^*-2}(\psi-T_\la(\psi),\phi)d\vol_\ig.
\endaligned
\]
Now $\cf_\la$ is $\cc^2$ because $T_\la$ is $\cc^1$.

d) For $\psi_0,\psi_1\in L^{2^*}$ and $\al\in[0,1]$ we obtain using the definition of $T_\la$:
\[
\aligned
  &\big|(1-\al)\psi_0+\al\psi_1-T_\la\big((1-\al)\psi_0+\al\psi_1\big)\big|_{2^*}^{2^*}\\
  &\hspace{1cm} \le \big|(1-\al)\psi_0+\al\psi_1-\big((1-\al)T_\la(\psi_0)+\al T_\la(\psi_1)\big)\big|_{2^*}^{2^*}\\
  &\hspace{1cm} \le (1-\al)\big|\psi_0-T_\la(\psi_0)\big|_{2^*}^{2^*}+\al\big|\psi_1-T_\la(\psi_1)\big|_{2^*}^{2^*}
\endaligned
\]

e) Observe that $\ce_\la'(\psi-T_\la(\psi))$ vanishes on $E_\la^0$ by the definition of $T_\la$. If $\psi$ is a critical point of $\tce_\la$ then $\ce_\la'(\psi-T_\la(\psi))$ vanishes also on $\te$.
\end{proof}

The second derivative of $\cf_\la$ is given by
\[
\aligned
  &\cf_\la''(\psi)[\phi,\chi]=\real\int_M|\psi-T_\la(\psi)|^{2^*-2}(\chi-T_\la'(\psi)[z],\phi)d\vol_\ig     \\
  &\qquad +\frac2{m-1}\int_M|\psi-T_\la(\psi)|^{2^*-4}\real(\psi-T_\la(\psi), \chi-T_\la'(\psi)[z])\cdot\real(\psi-T_\la(\psi),\phi)d\vol_\ig.
\endaligned
\]
Using $T_\la'(\psi)[\psi]=T_\la(\psi)$ and that $\cf_\la''(\psi)[\phi,\chi]=0$ for $\psi\in L^{2^*}\setminus E_\la^0$, $\phi\in L^{2^*}$ and $\chi\in E_\la^0$, an elementary calculation yields for $\psi\in L^{2^*}\setminus E_\la^0$ and $\phi\in L^{2^*}$:
\begin{\equ}\label{Tmfm-1}
\aligned
  &\big( \cf_\la''(\psi)[\psi,\psi]-\cf_\la'(\psi)[\psi]\big)+2\big(\cf_\la''(\psi)[\psi,\phi]-\cf_\la'(\psi)[\phi]\big) +\cf_\la''(\psi)[\phi,\phi] \\
  &\qquad \geq\frac2{m+1}\int_M|\psi-T_\la(\psi)|^{2^*}d\vol_\ig >0
\endaligned
\end{\equ}
Lemma~\ref{lem:prop-T} implies that $\tce_\la(\psi+\phi)=\widetilde\ce_\la(\psi)$ for $\psi\in E$ and $\phi\in E_\la^0$. Therefore we only need to consider $\tce_\la$ on $\te_\la=E_\la^+\op E_\la^-$, so from now on $\tce_\la:\te_\la\to\R$. A straightforward calculation shows for any $\psi,\phi\in \te_\la$:
\[
  \widetilde\ce_\la'(\psi)[\phi] = \real\int_M\big(D\psi-\lm\psi-|\psi-T_\la(\psi)|^{2^*-2}(\psi-T_\la(\psi)),\phi\big)d\vol_\ig.
\]
Next we construct the Nehari-Pankov manifold for $\tce_\la$. We could refer to \cite{Szulkin-Weth:2010} as in section~\ref{sec:minmax} but we prefer a different 2-step procedure which will make the subsequent estimates more transparent.

\begin{Prop}\label{reduction-prop}
\begin{itemize}
  \item[\rm a)] There exists a $\cc^1$ map $\eta_\la: E_\la^+\to E_\la^-$ such that for $\psi\in\te_\la$:
    \[
      \widetilde\ce_\la'(\psi)[\chi]=0 \quad\text{for all $\chi\in E_\la^-$} \qquad\Longrightarrow\qquad \psi^-=\eta_\la(\psi^+)
    \]
    Moreover, $\eta_\la(\phi)$ maximizes $\tce_\la(\phi+\chi)$ over all $\chi\in E_\la^-$.
   \item[\rm b)] The functional $\cj_\la:E_\la^+\to\R$, $\cj_\la(\phi)=\tce_\la\big(\phi+\eta_\la(\phi) \big)$, satisfies:
    \[
       \cj_\la'(\phi)=0 \qquad\Longrightarrow\qquad \tce_\la'\big(\phi+\eta_\la(\phi) \big)=0
    \]
   \item[\rm c)] For every $\phi\in E_\la^+\setminus\{0\}$, the map $\cj_{\la,\phi}:\R\to\R$, $\cj_{\la,\phi}(t):=\cj_\la(t\phi)$, is of class $\cc^2$ and satisfies
    \[
      \cj_{\la,\phi}'(t)=0,\ t>0 \qquad\Longrightarrow\qquad \cj_{\la,\phi}''(t)<0.
    \]
   Moreover $\cj_{\la,\phi}(0)=\cj_{\la,\phi}'(0)=0$, $\cj_{\la,\phi}''(0)>0$.
\end{itemize}
\end{Prop}

\begin{proof}
For $\phi\in E_\la^+$ the map $\tce_{\la,\phi}:E_\la^-\to\R$ defined by
\[
   \tce_{\la,\phi}(\chi) = \tce_\la(\phi+\chi) = \frac12\big( \|\phi\|_\la^2-\|\chi\|_\la^2 \big) - \cf_\la(\phi+\chi).
\]
is strict concave because $\cf_\la$ is convex. Moreover, it is anti-coercive, hence it has a unique critical point $\eta_\la(\phi)$, which is a maximum point. That $\eta_\la:E_\la^+\to E_\la^-$ is of class $\cc^1$ follows from the implicit function theorem applied to the equation $D_\chi\tce_\la(\phi+\chi)=0$ which defines $\chi=\eta_\la(\phi)$. This proves a).

For the proof of b) recall that $\tce_\la'(\phi+\eta_\la(\phi))[\chi]=0$ for all $\chi\in E_\la^-$ by construction of $\eta_\la$. If $\cj_\la'(\phi)=0$ then $\tce_\la'(\phi+\eta_\la(\phi))[\chi]=0$ holds also for all $\chi\in E_\la^+$ as a simple calculation shows.

In order to see c) we compute $\cj_{\la,\phi}'(t) = \tce_\la'(t\phi+\eta_\la(t\phi))[\phi]$ which implies that $\cj_{\la,\phi}$ is $\cc^2$. The implication in c) is equivalent to:
\[
  \cj_\la'(\phi)[\phi]=0,\ \phi\ne0 \qquad\Longrightarrow\qquad \cj_\la''(\phi)[\phi,\phi]<0.
\]
This is a consequence of the following computation where we set $\psi=\phi+\eta_\la(\phi)$ and $\chi=\eta_\la'(\phi)[\phi]-\eta_\la(\phi)$, and use that $\tce_\la'(\psi)|_{E_\la^-}=0$.
\begin{\equ}\label{eee1}
\aligned
 &\cj_\la''(\phi)[\phi,\phi]
   = \tce_\la''(\psi)[\phi+\eta_\la'(\phi)[\phi],\phi] = \tce_\la''(\psi)[\psi+\chi,\psi+\chi]\\
 &\hspace{1cm} = \widetilde\ce_\la''(\psi)[\psi,\psi]+2\widetilde\ce_\la''(\psi)[\psi,\chi]+\widetilde\ce_\la''(\psi)[\chi,\chi] \\
 &\hspace{1cm} =\cj_\la'(\phi)[\phi]+\big(\cf_\la(\psi)[\psi]-\cf_\la''(\psi)[\psi,\psi]\big)+2\big(\cf_\la'(\psi)[\chi]-\cf_\la''(\psi)[\psi,\chi] \big) \\
 &\hspace{2cm} -\cf_\la''(\psi)[\chi,\chi]-\|\chi\|_\la^2\\
 &\hspace{1cm} \leq \cj_\la'(\phi)[\phi]-\|\chi\|_\la^2-\frac2{m+1}|\psi-T_\la(\psi)|_{2^*}^{2^*}.
\endaligned
\end{\equ}
Finally we have $\cj_{\la,\phi}(0)=\tce_\la(0)=0$, $\cj_{\la,\phi}'(0)=\tce_\la'(0)[\phi]=0$, and $\cj_{\la,\phi}''(0)=\tce_\la''(0)[\phi,\phi]>0$.
\end{proof}

The Nehari-Pankov manifold for $\tce_\la$ is defined as
\[
  \widetilde\cp_\la := \{\phi\in E_\la^+\setminus\{0\}:\, \cj_\la'(\phi)[\phi]=0\}.
\]
By Proposition \ref{reduction-prop} this is a smooth submanifold of codimension $1$ in $E_\la^+$, and it is a natural constraint for the problem of finding non-trivial critical points of $\cj_\la$.

\begin{Lem}\label{PS-esti1}
If $\psi_n\in\te_\la$ satisfies
\begin{\equ}\label{aaa1}
  \big\|\widetilde\ce_\la'(\psi_n)\big|_{E_\la^-}\big\|_\la = \sup_{v\in E_\la^-,\, \|v\|=1}\widetilde\ce_\la'(\psi_n)[v] = o_n(1)
\end{\equ}
then
\[
  \|\psi_n^--\eta_\la(\psi_n^+)\|_\la = O\big(\big\|\widetilde\ce_\la'(\psi_n)\big|_{E_\la^-}\big\|_\la\big).
\]
Moreover, if $(\psi_n)_n$ is a $(PS)_c$-sequence for $\tce_\la$ on $\te_\la$, then $(\psi_n^+)_n$ is a $(PS)_c$-sequence for $\cj_\la$ on $E_\la^+$.
\end{Lem}

\begin{proof}
For simplicity of notation we set $\zeta_n=\psi_n^++\eta_\la(\psi_n^+)$ and $\xi_n=\psi_n-\zeta_n=\psi_n^--\eta_\la(\psi_n^+)\in E_\la^-$. The we have by definition of $\eta_\la$:
\[
  0=\widetilde\ce'(\zeta_n)[\xi_n]=-\inp{\eta_\la(\psi_n^+)}{\xi_n}-\real\int_M |\zeta_n-T_\la(\zeta_n)|^{2^*-2}(\zeta_n-T_\la(\zeta_n),\xi_n)d\vol_\ig.
\]
Next \eqref{aaa1} implies
\begin{\equ}\label{aaa2}
  o(\|\xi_n\|_\la) = \tce_\la'(\psi_n)[\xi_n]
   = -\inp{\psi_n^-}{\xi_n}_\la-\real\int_M |\psi_n-T(\psi_n)|^{2^*-2}(\psi_n-T_\la(\psi_n),\xi_n)d\vol_\ig.
\end{\equ}
Therefore we have
\begin{\equ}\label{aaa3}
\aligned
  o(\|\xi_n\|_\la)
    &=\|\xi_n\|^2+\real\int_M |\psi_n-T_\la(\psi_n)|^{2^*-2}(\psi_n-T_\la(\psi_n),\xi_n)d\vol_\ig\\
    &\qquad -\real\int_M |\zeta_n-T_\la(\zeta_n)|^{2^*-2}(\zeta_n-T_\la(\zeta_n),\xi_n)d\vol_\ig.
\endaligned
\end{\equ}
Since the functional $\psi\mapsto |\psi-T_\la(\psi)|_{2^*}^{2^*}$ is convex, we obtain
\[
\aligned
  &\real\int_M |\psi_n-T_\la(\psi_n)|^{2^*-2}(\psi_n-T_\la(\psi_n),\xi_n)d\vol_\ig\\
  &\hspace{1cm} -\real\int_M |\zeta_n-T_\la(\zeta_n)|^{2^*-2}(\zeta_n-T_\la(\zeta_n),\xi_n)d\vol_\ig \geq 0.
\endaligned
\]
Now \eqref{aaa2}, \eqref{aaa3} and $\xi_n\in E_\la^-$ yield $\|\xi_n\|_\la = O\big(\big\|\widetilde\ce_\la'(\psi_n)\big|_{E_\la^-}\big\|_\la\big)$.

If $(\psi_n)_n$ is a $(PS)_c$-sequence for $\tce_\la$ then $\big(\psi_n-T_\la(\psi_n)\big)_n$ is a $(PS)_c$-sequence for $\ce_\la$, hence it is bounded by Lemma~\ref{PS-bdd}. Now a bound on the second derivative of $\ce_\la$ implies that $(\zeta_n)_n$ is a $(PS)_c$-sequence for $\tce_\la$, hence $(\psi_n^+)_n$ is a $(PS)_c$-sequence for $\cj_\la$ on $E_\la^+$.
\end{proof}

It is not difficult to check that the functional
\[
  \ch_\la: E_\la^+\to\R,\quad \ch_\la(\phi)=\cj_\la'(\phi)[\phi],
\]
is of class $\cc^1$ with derivative
\[
  \ch_\la'(\phi)[\chi] = \cj_\la'(\phi)[\chi]+\cj_\la''(\phi)[\phi,\chi]
\]
for $\phi,\,\chi\in E_\la^+$. Observe that $\widetilde\cp_\la=\ch_\la^{-1}(0)\setminus\{0\}$.

\begin{Lem}\label{est-G}
For $\phi\in E_\la^+$ and $\psi:=\phi+\eta_\la(\phi)$ there holds
\[
  \ch_\la'(\phi)[\phi] \leq 2\ch_\la(\phi)-\frac2{m+1}\big| \psi-T_\la(\psi) \big|_{2^*}^{2^*}.
\]
\end{Lem}

\begin{proof}
This estimate follows immediately from \eqref{Tmfm-1} and a similar argument as in \eqref{eee1}.
\end{proof}

\begin{Prop}\label{PS-esti2}
For any $c>0$, if $(\phi_n)_n$ is a $(PS)_c$-sequence for $\cj_\la$ then there exists a sequence $(t_n)_n$ in $\R$ such that $t_n\phi_n\in\widetilde\cp_\la$ and $|t_n-1|=(\|\cj_\la'(\phi_n)\|_\la)$.
\end{Prop}

\begin{proof}
If $(\phi_n)_n$ is a $(PS)_c$-sequence for $\cj_\la$ then $\big(\psi_n=\phi_n+\eta(\phi_n)\big)_n$ is a $(PS)_c$-sequence for $\tce_\la$, hence $(\psi_n-T_\la(\psi_n))_n$ is a $(PS)_c$-sequence for $\ce_\la$ which is bounded by Lemma \ref{PS-bdd}. Therefore $(\phi_n)_n$ is bounded. Moreover, since $\cj_\la(\phi_n)\to c>0$ we obtain:
\[
  \liminf_{n\to\infty}\big|\psi_n-T_\la(\psi_n)\big|_{2^*} > 0.
\]
Now we define $g_n:(0,+\infty)\to\R$ by $g_n(t)=\ch_\la(t\phi_n)$. Then $tg_n'(t)=\ch_\la'(t\phi_n)[t\phi_n]$, hence, by Lemma~\ref{est-G}, Taylor's formula and the uniform boundedness of $g_n'(t)$ on bounded intervals, we get
\[
  tg_n'(t)\leq 2g_n(1)-\frac2{m+1}\big| \psi_n-T_\la(\psi_n) \big|_{2^*}^{2^*}+C|t-1|
\]
for $t$ close to $1$ and some $C>0$ independent of $n$. Since $(\phi_n)_n$ is a $(PS)$-sequence for $\cj_\la$, we have $g_n(1)=\cj_\la'(\phi_n)[\phi_n]\to0$. Therefore there exists a constant $\de>0$ such that
\[
  g_n'(t)<-\de \text{ for all } t\in(1-\de,1+\de) \text{ and } n \text{ large.}
\]
Moreover, since $g_n(1-\de)>0$ and $g_n(1+\de)<0$ the Inverse Function Theorem yields that
$
 \bar \phi_n:=g_n^{-1}(0)\phi_n\in\widetilde\cp_\la\cap \span\{\phi_n\}
$
is well-defined for $n$ large. Furthermore, $g_n'(t)^{-1}$ is bounded by a constant, say, $c_1>0$ on $(1-\de,1+\de)$ due to the boundedness of $\{\phi_n\}$. As a consequence
\[
  \|\phi_n-\bar \phi_n\|_\la = |g_n^{-1}(0)-1|\cdot\|\phi_n\|_\la
   = \big|g_n^{-1}(0)-g_n^{-1}(\ch_\la(\phi_n))\big|\cdot\|\phi_n\|_\la
   \leq c_1\cdot |\ch_\la(\phi_n)|\cdot\|\phi_n\|_\la.
\]
Now the conclusion follows from
$|\ch_\la(\phi_n)| = O(\|\cj_\la'(\phi_n)\|_\la)$.
\end{proof}

Combining Lemma \ref{PS-esti1} and Proposition \ref{PS-esti2}, we obtain

\begin{Cor}\label{PS-esti-key}
For any $c>0$, if $(\psi_n)_n$ is a $(PS)_c$-sequence for $\tce_\la$, then there exists a sequence $(\phi_n)_n$ in $\widetilde\cp_\la$ such that $\|\psi_n-\phi_n-\eta_\la(\phi_n)\|_\la = O(\|\widetilde\ce_\la'(\psi_n)\|_\la)$.
In particular,
\[
  \max_{t>0}\cj_\la(t\psi_n^+) = \cj_\la(\phi_n) \leq \widetilde\ce_\la(\psi_n)+O(\|\widetilde\ce_\la'\big(\psi_n)\|_\la^2\big).
\]
\end{Cor}

\begin{proof}
According to Lemma \ref{PS-esti1}, we have
$
 \|\psi_n^--\eta_\la(\psi_n^+)\|_\la \leq O(\|\widetilde\ce_\la'(\psi_n)\|_\la)
$
and $(\psi_n^+)_n$ is a $(PS)_c$-sequence for $\cj_\la$. Then, by Proposition \ref{PS-esti2}, there exists $\phi_n=t_n\psi_n^+$
such that
\begin{equation}\label{aaa4}
\begin{aligned}
\|\psi_n-\phi_n-\eta_\la(\phi_n)\|_\la
 &\leq \|\psi_n^--\eta_\la(\psi_n^+)\|_\la+\|\psi_n^+-\phi_n\|_\la+\|\eta_\la(\psi_n^+)-\eta_\la(\phi_n)\|_\la \\
 &\leq O(\|\widetilde\ce_\la'(\psi_n)\|_\la)+ O(\|\cj_\la'(\psi_n^+)\|_\la).
\end{aligned}
\end{equation}
Here we used that $(\psi_n)_n$ is bounded due to Lemma \ref{PS-bdd}, and the inequality
\[
  \|\eta_\la(\psi_n^+)-\eta_\la(\phi_n)\|_\la \leq \|\eta_\la'(\tau_n\psi_n^+)\|_\la\cdot \|\psi_n^+-\phi_n\|_\la = O(|t_n-1|)
\]
which can be easily checked. The boundedness of the second derivative of $\widetilde\ce_\la$ and Lemma \ref{PS-esti1} yield
\[
  \|\cj_\la'(\psi_n^+)\|_\la = \|\widetilde\ce_\la'(\psi_n^++\eta_\la(\psi_n^+))\|_\la
   \leq \|\widetilde\ce_\la'(\psi_n)\|_\la+O(\|\psi_n^--\eta_\la(\psi_n^+)\|_\la) = O(\|\widetilde\ce_\la'(\psi_n)\|_\la).
\]
This together with \eqref{aaa4} implies $\|\psi_n-\phi_n-\eta_\la(\phi_n)\|_\la = O(\|\widetilde\ce_\la'(\psi_n)\|_\la)$.

Next, Taylor's formula and the boundedness of the second derivative of $\widetilde\ce_\la$ imply
\[
\aligned
\widetilde\ce_\la(\psi_n)
 &= \widetilde\ce(\phi_n+\eta(\phi_n))+\widetilde\ce'(\phi_n
      +\eta_\la(\phi_n))[\psi_n-\phi_n-\eta_\la(\phi_n)]+O(\|\widetilde\ce_\la'(\psi_n)\|_\la^2)\\
 &=\cj_\la(\phi_n)+\cj_\la'(\phi_n)[\psi_n^+-\phi_n]+O(\|\widetilde\ce_\la'(\psi_n)\|_\la^2).
\endaligned
\]
Finally we have $\cj_\la'(\phi_n)[\psi_n^+-\phi_n]=0$ because $\phi_n\in\widetilde\cp_\la\cap \span\{\psi_n^+\}$. This implies the last estimate of the corollary.
\end{proof}

Now we address the main inequality \eqref{target} using our test spinor $\bar\va_\vr$ in \eqref{test spinor}. Clearly we have
\begin{\equ}\label{l1-norm}
  |\bar\va_\vr|_1=\int_M|\bar\va_\vr|d\vol_\ig \lesssim \vr^{\frac{m-1}2}
  \qquad\text{and}\qquad
  |\bar\va_\vr|_{2^*-1}^{2^*-1}=\int_M|\bar\va_\vr|^{2^*-1}d\vol_\ig \lesssim \vr^{\frac{m-1}2}
\end{\equ}
Since $\dim E_\la^0<\infty$ we obtain
\begin{\equ}\label{l-infty-norm}
  |T_\la(\bar\va_\vr)|_1 \lesssim |\bar\va_\vr|_1 \lesssim \vr^{\frac{m-1}2}
  \qquad\text{and}\qquad
  |T_\la(\bar\va_\vr)|_2 \lesssim |T_\la(\bar\va_\vr)|_1 \lesssim \vr^{\frac{m-1}2}.
\end{\equ}
It follows that
\begin{equation}\label{test-T-1}
\begin{aligned}
&\Big| \int_M|\bar\va_\vr-T_\la(\bar\va_\vr)|^{2^*}-|\bar\va_\vr|^{2^*} d\vol_\ig \Big|
  \lesssim \Big|\int_M|\bar\va_\vr-\theta T_\la(\bar\va_\vr)|^{2^*-2}(\bar\va_\vr-\theta T_\la(\bar\va_\vr), T(\bar\va_\vr))\,d\vol_\ig \Big| \\
&\hspace{1cm}
  \leq \int_M|\bar\va_\vr-\theta T_\la(\bar\va_\vr)|^{2^*-1}\cdot|T_\la(\bar\va_\vr)|\,d\vol_\ig
  \lesssim \int_M |\bar\va_\vr|^{2^*-1}|T_\la(\bar\va_\vr)|+|T_\la(\bar\va_\vr)|^{2^*} d\vol_\ig  \\
&\hspace{1cm}
  \lesssim |T_\la(\bar\va_\vr)|_\infty \int_M|\bar\va_\vr|^{2^*-1}d\vol_\ig + |T_\la(\bar\va_\vr)|_\infty^{2^*}
  \lesssim \vr^{m-1}
\end{aligned}
\end{equation}
where $0<\theta<1$. Moreover, for any $\psi\in E$ with $\|\psi\|_\la\leq1$, we can deduce from \eqref{l1-norm} and \eqref{l-infty-norm} that
\begin{eqnarray}\label{test-T-2}
&&\Big| \int_M|\bar\va_\vr-T_\la(\bar\va_\vr)|^{2^*-2}
(\bar\va_\vr-T_\la(\bar\va_\vr),\psi) -
|\bar\va_\vr|^{2^*-2}(\bar\va_\vr,\psi) d\vol_\ig \Big|
\nonumber\\
&& \qquad \qquad\leq
(2^*-1)\int_M|\bar\va_\vr-\theta
T_\la(\bar\va_\vr)|^{2^*-2}\cdot
|T_\la(\bar\va_\vr)|\cdot |\psi| d\vol_\ig  \nonumber \\
&& \qquad \qquad\lesssim
\int_M|\bar\va_\vr|^{2^*-2}\cdot|T_\la(\bar\va_\vr)|\cdot
|\psi|+ |T_\la(\bar\va_\vr)|^{2^*-1}\cdot|\psi| d\vol_\ig
\nonumber  \\
&& \qquad \qquad \lesssim
|T_\la(\bar\va_\vr)|_\infty \cdot \Big(\int_M|\bar\va_\vr|^{\frac{4m}{m^2-1}}d\vol_\ig\Big)^{\frac{m+1}{2m}}
+ |T_\la(\bar\va_\vr)|_\infty^{2^*-1}   \nonumber\\
&& \qquad \qquad \lesssim \ \begin{cases}
 \vr  &\text{if } m=2, \\
 \vr^2|\ln\vr|^{\frac23} &\text{if } m=3,\\
 \vr^{\frac{m+1}2} &\text{if } m\geq4.
 \end{cases}
\end{eqnarray}
Therefore $(\widetilde\va_\vr:=\bar\va_\vr-T_\la(\bar\va_\vr))_\vr$ is a $(PS)$-sequence for $\tce_\la$ on $E$. Combining these facts and Lemma~\ref{derivative-esti}, Lemma~\ref{energy-esti}, Corollary~\ref{PS-esti-key}, we finally obtain
\begin{\equ}\label{energy-esti-key-2}
\max_{t>0}\cj_\la(t\bar\va_\vr^+)
 \leq \frac1{2m}\left(\frac{m}2\right)^m \om_m
    + \begin{cases}
        -\lm C \vr|\ln\vr| + O(\vr)  &\text{if } m=2,\\
        -\lm C \vr + O(\vr^2|\ln\vr|^{\frac43}) &\text{if } m=3,\\
        -\lm C \vr+ O(\vr^2) &\text{if } m\geq4
      \end{cases}
\end{\equ}
for some constant $C>0$ depending only on the dimension.

Now \eqref{energy-estimate-M} implies, setting $\phi_\vr=\frac{1}{\|\bar\va_\vr^+\|_\la}\bar\va_\vr^+$,
\[
  \inf\cm_\la \leq \max_{\psi\in \ehat_\la(\phi_\vr)}\tce_\la(\psi) = \max_{t>0}\cj_\la(t\bar\va_\vr^+)
    < \frac1{2m}\left(\frac{m}2\right)^m \om_m = \ga_{crit}
\]
for $\vr>0$ small.

This concludes the proof of Theorem~\ref{main thm}. As mentioned before Lemma~\ref{lem:continuation} also Theorem~\ref{thm:mult} follows.

\subsection{Proof of Theorem \ref{main thm} for $\la\leq0$}
This situation is much simpler because hypothesis $(f_5)$ is already an estimate of $\int_M F(|\bar\va_\vr|)d\vol_\ig$ needed below. Indeed, first of all, we have
\[
\aligned
\vr^m\int_{|y|\leq\frac\de\vr}F\bigg( \frac{A\vr^{-\frac{m-1}2}}{(1+|y|^2)^{\frac{m-1}2}} \bigg)dy
  &= \vr^m\om_{m-1}\int_0^{\frac\de\vr}F\bigg( \frac{A\vr^{-\frac{m-1}2}}{(1+r^2)^{\frac{m-1}2}} \bigg)r^{m-1}dr  \\
  &= \de^m\vr^m\om_{m-1}\int_0^{\frac1\vr}F\bigg( \frac{A\vr^{-\frac{m-1}2}}{(1+\de^2s^2)^{\frac{m-1}2}} \bigg)s^{m-1}ds
\endaligned
\]
where $r=\de s$ and $A>0$ is some constant. Thus, after rescaling $\vr$, hypothesis $(f_5)$ and Lemma~\ref{energy-esti} imply
\begin{\equ}\label{F-lower esti}
\frac{\int_MF(|\bar\va_\vr|)d\vol_\ig}{\int_M|\bar\va_\vr|^2d\vol_\ig}
  \geq \frac{C\cdot\vr^{m-1}}{|\ln\vr|^{\max\{3-m,\,0\}}}\int_0^{\frac1\vr}F\bigg( \frac{\vr^{-\frac{m-1}2}}{(1+s^2)^{\frac{m-1}2}} \bigg) s^{m-1}ds
  \to \infty
\end{\equ}
as $\vr\to0$ for some constant $C>0$.

According to hypothesis $(f_4)$, for any $\de>0$ there is $C_\de>0$ such that $f(s)s\leq C_\de s+\de F(s)^{\frac{m+1}{2m}}$ for all $s\geq0$. Therefore we have
\[
\|f(|\bar\va_\vr|)\bar\va_\vr\|_{E^*}
  \lesssim C_\de \|\bar\va_\vr\|_{E^*} + \de \Big( \int_MF(|\bar\va_\vr|)d\vol_\ig \Big)^{\frac{m+1}{2m}}.
\]

Since $\cl_\la$ has the form
\[
  \cl_\la(\psi) = \frac12\left(\|\psi^+\|_\la^2-\|\psi^-\|_\la^2\right) - \ck(\psi)
\]
with $\ck(\psi) = \int_M F(|\psi|)d\vol_\ig + \frac1{2^*}\int_M|\psi|^{2^*}d\vol_\ig$ being strictly convex, a straightforward calculation shows that $\ck$ also satisfies the following inequality which is an analogue of \eqref{Tmfm-1}. For any $\psi\in E\setminus\{0\}$ and $\phi\in E$ there holds:
\[
  \big( \ck''(\psi)[\psi,\psi]-\ck'(\psi)[\psi] \big)+2\big( \ck''(\psi)[\psi,\phi]-\ck'(\psi)[\phi] \big)+\ck''(\psi)[\phi,\phi]
\geq \frac2{m+1}|\psi|_{2^*}^{2^*} \,.
\]
Therefore, for $\lm\not\in \spec(D)\cap(-\infty,0]$ Lemma~\ref{reduction-prop} applies to $\cl_\la$, and we can use the arguments following it to conclude that
\[
  \inf\cm_\la \leq \max_{\psi\in \ehat(\phi_\vr)}\cl_\la(\psi) \leq \cl_\la(\bar\va_\vr)+O(\|\cl_\la'(\bar\va_\vr)\|_\la^2)
\]
where $\phi_\vr=\frac{1}{\|\bar\va_\vr^+\|_\la}\bar\va_\vr^+$ and $\bar\va_\vr$ is our test spinor.

Now we deduce from Lemma \ref{derivative-esti} and Lemma \ref{energy-esti} that
\[
\aligned
\cl_\la(\bar\va_\vr)+O(\|\cl_\la'(\bar\va_\vr)\|_\la^2)
  &\leq \frac1{2m}\big(\frac{m}2\big)^m \om_m-\int_MF(|\bar\va_\vr|)d\vol_\ig
          -\begin{cases}
             \lm C_\de' \vr|\ln\vr| &\text{if } m=2\\
             \lm C_\de'\vr &\text{if } m\geq3
           \end{cases}  \\
  &\qquad + \de^2 \cdot C \Big( \int_MF(|\bar\va_\vr|)d\vol_\ig \Big)^{\frac{m+1}{m}}.
\endaligned
\]
This together with \eqref{F-lower esti} implies
\[
  \max_{\psi\in \ehat_\la(\phi_\vr)}\cl_\la(\psi) < \frac1{2m}\big(\frac{m}2\big)^m \om_m \quad
  \text{for } \vr \text{ small}.
\]
Now the existence result follows easily, completing the proof of Theorem~\ref{main thm}.

\subsection{Proof of Theorem~\ref{thm:mult2}}

To begin with we consider the functional $\rr_\la: \te_\la\setminus\{0\}\to\R$ defined by
\[
  \rr_\la(\psi) = \frac{\int_M(D\psi,\psi)-\la|\psi|^2 d\vol_\ig}{\big(\int_M|\psi-T_\la(\psi)|^{2^*}d\vol_\ig\big)^{\frac2{2^*}}}
                    = \frac{\|\psi^+\|_\la^2-\|\psi^-\|_\la^2}{|\psi-T_{\la}(\psi)|_{2^*}^2}
\]
Then we have
\begin{\equ}\label{eq:R-derivative}
  \rr_\la'(\psi)[\va] = \frac{2}{A_\la(\psi)^{\frac2{2^*}}}\left( \real\int_M((D-\la)\psi,\va)d\vol_\ig-\frac{\rr_\la(\psi)}{2^*}A_\la(\psi)^{\frac{2-2^*}{2^*}}\cdot A_\la'(\psi)[\va] \right)
\end{\equ}
where
\[
  A_\la(\psi)=\int_M|\psi-T_\la(\psi)|^{2^*} d\vol_\ig
\]

For any $\phi\in E_\la^+\setminus\{0\}$ and  any $c>0$ the set $\{\chi\in E_\la^-: \rr_\la(\phi+\chi)\ge c\}$ is strictly convex and bounded because
\[
  \chi\mapsto \|\phi\|_\la^2-\|\chi\|_\la^2-\la|\phi+\chi|_2^2-c\big|\phi+\chi-T_\la(\phi+\chi)\big|_{2^*}^{2^*}
\]
is strictly concave and anti-coercive on $E_\la^-$. Therefore the map $\chi\mapsto \rr_\la(\phi+\chi)$ has a unique maximum point $\chi_\phi\in E_\la^-$. Now, let us define
\[
  \cs_\la(\phi) = \rr_\la(\phi+\chi_\phi)
\]

\begin{Lem}\label{identity cj}
  $\cs_\la(\phi)=\big( 2m \cj_\la(\phi)\big)^{\frac 1m}$ for $\phi\in\widetilde\cp_\la$.
\end{Lem}

\begin{proof}
Let $\phi\in\widetilde\cp_\la$, then
\[
0=\cj_\la'(\phi)[\phi]=\int_M\big(D(\phi+\eta_\la(\phi)),  \phi+\eta_\la(\phi)\big)d\vol_\ig-\la|\phi+\eta_\la(\phi)|_2^2-\big|\phi+\eta_\la(\phi)-T_\la(\phi+\eta_\la(\phi)) \big|_{2^*}^{2^*}
\]
Hence $\cj_\la(\phi)=\cj_\la(\phi)-\frac12\cj_\la'(\phi)[\phi]=\frac1{2m}\big|\phi+\eta_\la(\phi)-T_\la(\phi+\eta_\la(\phi)) \big|_{2^*}^{2^*}$.

On the other hand, \eqref{eq:R-derivative} implies
\[
  \rr_\la'(\phi+\eta_\la(\phi))[\chi] \equiv 0 \quad \text{for all $\chi\in E_\la^-$.}
\]
This together with the fact $\rr_\la(\phi+\eta_\la(\phi))>0$ yields that $\chi_\phi=\eta_\la(\phi)$ for $\phi\in\widetilde\cp_\la$. And this in turn implies
\[
  \cs_\la(\phi) = \left( \int_M\big|\phi+\eta_\la(\phi)-T_\la(\phi+\eta_\la(\phi)) \big|^{2^*} \right)^{1-\frac2{2^*}}
                    = \big( 2m \cj_\la(\phi)\big)^{\frac 1m}
\]
\end{proof}

Theorem~\ref{thm:mult2} follows from the next Proposition. Recall the definition of
\[
  \nu = \frac{m}2\bigg( \frac{\om_m}{ {\rm Vol} (M,\ig)} \bigg)^{\frac1m}
\]
from \eqref{eq:nu}.

\begin{Prop}
For $\la\in\R$ the functional $\cm_\la$ admits at least 
\[
  \ell = \ell(\la) = \dim_\C\left( \bigoplus_{\la<\la_k<\la+\nu}\ker(D-\la_k) \right).
\]
distinct $S^1$-orbits of critical points.
\end{Prop}

\begin{proof}
Let $\ga=\ga_{S^1}$ denote the $S^1$-genus, i.e.\ for a topological space $X\ne\emptyset$ on which the group $S^1$ acts continuously without fixed points, $\ga(X)$ is the infimum over all $k\in\N$ such that there exist finite subgroups $H_1,\dots,H_k\subset S^1$ and a continuous equivariant map $X \to S^1/H_1\ast\dots\ast S^1/H_k$ where $\ast$ denotes the join. If no such $k$ exists this means $\ga(X)=\infty$. The $S^1$-genus has properties analogous to the Krasnoselski genus for spaces with an action of $Z/2$; see \cite{Bartsch:1990, Bartsch:1993}. Now we define
\[
  \be_j := \inf_{\substack{A\subset S_\la^+ \\ \ga(A)\geq j}}\max_{\phi\in A}\cm_\la(\phi)\quad \text{for $j\geq1$.}
\]
where we only consider $S^1$-invariant subsets $A\subset S_\la^+$. Clearly we have $\be_j\le\be_{j+1}$ for all $j\ge1$. If
\[
  X_j \subset \bigoplus_{\la<\la_k<\la+\nu}\ker(D-\la_k)
\]
is any complex $j$-dimensional subspace, $1\le j\le\ell$, then $X_j\cap S_\la^+$ is equivariantly homeomorphic to the unit sphere in $X_j$, hence $\ga(X_j\cap S_\la^+) = \dim_\C(X_j)$. Now \eqref{def:tilde-e} and Lemma \ref{identity cj} imply
\[
  \be_j \le \max_{\phi\in X_j\cap S_\la^+}\cm_\la(\phi)\le \max_{\phi\in X_j\cap \widetilde\cp_\la}\cj_\la(\phi)
             = \frac 1{2m}\max_{\phi\in X_j\cap \widetilde\cp_\la} S_\la(\phi)^m.
\]
Observe that, for any $\phi\in X_j\cap \widetilde\cp_\la$,
\[
  \|\phi\|_\la^2-\|\chi_\phi\|_\la^2 < \|\phi\|_\la^2 \leq \nu|\phi|_2^2 < \nu|\phi+\chi_\phi-T_\la(\phi+\chi_\phi)|_2^2
\]
where the last inequality follows from the fact that $\phi\in E_\la^+$, $\chi_\phi\in E_\la^-$ and $T_\la(\phi+\chi_\phi)\in E_\la^0$ are orthogonal in $L^2$. Finally, using H\"older's inequality, we find for $j=1,\dots,\ell$:
\[
  \be_j < \frac{\nu^m}{2m} \cdot \max_{\phi\in X_j\cap \widetilde\cp_\la}
                    \left(\frac{|\phi+\chi_\phi-T_\la(\phi+\chi_\phi)|_2^2}{|\phi+\chi_\phi-T_\la(\phi+\chi_\phi)|_{2^*}^2}\right)^m
          \le \nu^m\cdot\frac{{\rm Vol} (M,\ig)}{2m}=\ga_{crit}.
\]
Since the Palais-Smale condition holds below $\ga_{crit}$ each $\be_j$, $j=1,\dots,\ell$, is a critical value, and if $\be_j=\be_{j+1}$ for some $j$ then $\cm_\la$ has infinitely many $S^1$-orbits of critical points at the level$\be_j$; see \cite[Theorem~2.19]{Bartsch:1993}.
\end{proof}

\vspace{2mm}
{\sc Thomas Bartsch\\
 Mathematisches Institut, Universit\"at Giessen\\
 35392 Giessen, Germany}\\
 Thomas.Bartsch@math.uni-giessen.de\\

{\sc Tian Xu\\
 Center for Applied Mathematics, Tianjin University\\
 Tianjin, 300072, China}\\
 xutian@amss.ac.cn

\end{document}